\documentclass[11pt]{amsart}
\usepackage{amssymb, amstext, amscd, amsmath, amssymb}
\usepackage{mathtools, paralist, color, dsfont, rotating}
\usepackage{verbatim}
\usepackage{enumerate,enumitem}
\usepackage{float}
\usepackage{setspace}
\usepackage{pifont}

\usepackage{tikz}
\usetikzlibrary{fit,positioning,arrows,automata,calc, graphs}
\tikzset{
  main/.style={circle, minimum size = 30pt, thick, draw = black!80, node distance = 10mm},
  connect/.style={-latex, thick},
  box/.style={rectangle, draw = white!100}
}
\usepackage{tikz-cd}

\usepackage{graphicx}

\numberwithin{equation}{section}
\usepackage[a4paper]{geometry}
\usepackage{quoting}
\quotingsetup{vskip=.1in}
\quotingsetup{leftmargin=.6in} 
\quotingsetup{rightmargin=.6in}

\let\OLDthebibliography\thebibliography
\renewcommand\thebibliography[1]{
  \OLDthebibliography{#1}
  \setlength{\parskip}{0pt}
  \setlength{\itemsep}{2pt plus 0.5ex}
}

\makeatletter
\def\@cite#1#2{{\m@th\upshape\bfseries%
[{#1\if@tempswa{\m@th\upshape\mdseries, #2}\fi}]}}
\makeatother
\theoremstyle{plain}
\newtheorem{theorem}{Theorem}[section]
\newtheorem{corollary}[theorem]{Corollary}
\newtheorem{proposition}[theorem]{Proposition}
\newtheorem{lemma}[theorem]{Lemma}
\theoremstyle{definition}
\newtheorem{definition}[theorem]{Definition}
\newtheorem{example}[theorem]{Example}

\newtheorem{remark}[theorem]{Remark}

\theoremstyle{remark}

\mathtoolsset{centercolon}
  \newcommand{\A}{{\mathcal{A}}}
  \newcommand{\B}{{\mathcal{B}}}
  \newcommand{\C}{{\mathcal{C}}}

  \newcommand{\F}{{\mathcal{F}}}
  \newcommand{\G}{{\mathcal{G}}}
\renewcommand{\H}{{\mathcal{H}}}

  \newcommand{\K}{{\mathcal{K}}}

  \newcommand{\M}{{\mathcal{M}}}
  \newcommand{\N}{{\mathcal{N}}}

\renewcommand{\P}{{\mathcal{P}}}
  
  \newcommand{\R}{{\mathcal{R}}}
\renewcommand{\S}{{\mathcal{S}}}
  \newcommand{\T}{{\mathcal{T}}}

  \newcommand{\X}{{\mathcal{X}}}
  \newcommand{\Y}{{\mathcal{Y}}}
  
\newcommand{\eps}{\varepsilon}
\def\al{\alpha}
\def\be{\beta}

\def\ga{\gamma}
\def\De{\Delta}
\def\de{\delta}
\def\ze{\zeta}
\def\io{\iota}

\def\la{\lambda}

\newcommand\vphi{\varphi}

\newcommand{\bC}{\mathbb{C}}

\newcommand{\bN}{\mathbb{N}}

\newcommand{\fC}{{\mathfrak{C}}}

\newcommand{\fH}{{\mathfrak{H}}}

\newcommand{\foral}{\text{ for all }}
\newcommand{\qand}{\quad\text{and}\quad}

\newcommand{\ca}{\mathrm{C}^*}

\newcommand{\cenv}{\mathrm{C}^*_{\textup{env}}}

\newcommand{\ol}{\overline}
\newcommand{\wt}{\widetilde}
\newcommand{\wh}{\widehat}

\newcommand{\ad}{\operatorname{ad}}

\newcommand{\ran}{\operatorname{ran}}

\newcommand{\spn}{\operatorname{span}}

\newcommand{\sca}[1]{\left\langle#1\right\rangle} 
\newcommand{\bo}[1]{\mathbf{#1}} 

\newcommand{\tes}[7]{
	\xymatrix@C=2cm@R=1.5cm{
		K_0\left(#1\right) \ar[r]^{#2} & K_0\left(#3\right) \ar[r]^{#4} & K_0\left(#5\right) \ar[d] \\
		K_1\left(#5\right) \ar[u] & K_1\left(#3\right) \ar[l]^{#7} & K_1\left(#1\right) \ar[l]^{#6}
	}
}

\addtocontents{toc}{\protect\setcounter{tocdepth}{1}}

\begin{document}

\title[Morita equivalence for quantum graphs]{Morita equivalence for  quantum graphs}

\author[A. Chatzinikolaou]{Alexandros Chatzinikolaou}
\address{Institut Mittag-Leffler\\ The Royal Swedish Academy of Sciences \\ Djursholm\\ Sweden}
\email{achatzinik.math@gmail.com}

\author[G. Hoefer]{Gage Hoefer}
\address{Department of Mathematics \\ Dartmouth College \\ Hanover \\ NH 03755 \\ USA}\email{gage.hoefer@dartmouth.edu}

\author[N. Koutsonikos-Kouloumpis]{Nikolaos Koutsonikos-Kouloumpis}
\address{Department of Mathematics\\ University of Patras\\
Patras\\ 265 04\\ Greece}
\email{nkoutsonik@ac.upatras.gr}

\author[I.A. Paraskevas]{Ioannis Apollon Paraskevas}
\address{Department of Mathematics\\ National and Kapodistrian University of Athens\\ Athens\\ 1578 84\\ Greece}
\email{ioparask@math.uoa.gr}

\thanks{2020 {\it  Mathematics Subject Classification.} 16D90, 46L07, 47L25, 81P45}

\thanks{{\it Key words and phrases:} Quantum graphs, Morita equivalence, ternary rings of operators, operator systems, graph parameters.}

\date{\today}

\begin{abstract}
We introduce an operator-algebraic framework for Morita equivalence of quantum graphs based on $\Delta$-equivalence of operator systems introduced by Eleftherakis, Kakariadis and Todorov. Adopting the perspective of Weaver, we view quantum graphs as quantum relations, that is, operator systems endowed with a bimodule structure over the commutant of a von Neumann algebra. Within this framework, we show that two irreducibly acting quantum graphs are Morita equivalent if and only if they are both full pullbacks of a common quantum graph. This extends a result of Eleftherakis, Kakariadis and Todorov for graph operator systems to the quantum graph setting. In passing we construct a true-twin reduction analogue for an irreducibly acting quantum graph. We further characterise the case where we have simultaneous TRO-equivalence of the quantum graphs and their associated algebras, thus giving a second, stronger notion of Morita equivalence.
In the special case of noncommutative graphs, corresponding to the zero-error quantum communication setting, 
the two notions coincide and we obtain a characterisation in terms of strong co-homomorphisms of noncommutative graphs. Finally, we show that connectivity, the independence number, Shannon capacity, quantum complexity and subcomplexity, Haemers bound, and the Lov\'asz number are invariant under Morita equivalence. 
\end{abstract}

\maketitle


\section{Introduction}
\subsection{Quantum graphs as quantum relations}
Quantum graphs have attracted significant attention in recent years due to their deep connections with operator algebras, quantum information theory, quantum groups, and noncommutative geometry. They were first introduced by Duan, Severini and Winter \cite{DSW13} in the study of noncommutative confusability graphs of quantum channels, where they were modeled as operator systems $\S \subseteq M_n$. Around the same time, Weaver \cite{Wea12} introduced the notion of quantum graphs as quantum relations, and subsequently Musto, Reutter and Verdon \cite{MRV18} developed a categorical framework via quantum adjacency operators.  These approaches were later shown to be equivalent; see \cite{Daws24, Was24}.

Classically, a graph can be viewed as a relation $E \subseteq V \times V$ that is symmetric (undirected) and reflexive (with loops at every vertex). The noncommutative analogue of a relation is given by a bimodule over a maximal abelian selfadjoint algebra (MASA). This correspondence originates in the work of Erd{o}s, Katavolos and Shulman \cite{EKS98} and in the work of Shulman and Turowska \cite{ST04}. More precisely, given standard measure spaces $(X,\mu)$ and $(Y,\nu)$, every subset $\kappa \subseteq X \times Y$ determines a subspace $\M(\kappa) \subseteq \B(L^2(X,\mu), L^2(Y,\nu))$ consisting of operators supported on $\kappa$, and this yields an $L^\infty(X,\mu)$-$L^\infty(Y,\nu)$-bimodule. Conversely, every such MASA-bimodule arises in this way (see \cite[Theorems 4.1 and 4.2]{EKS98}). In finite dimensions, this correspondence becomes particularly concrete: a relation $\kappa \subseteq [n] \times [n]$ corresponds to the subspace
\[
\M(\kappa) = \operatorname{span}\{\varepsilon_{i,j} : (i,j) \in \kappa\} \subseteq M_n,
\]
which is a $D_n$-bimodule, and every $D_n$-bimodule is of this form \cite[Proposition 1.1]{PPS89}.

In this work, we adopt Weaver’s perspective \cite{Wea12, Wea21}, where they defined a \emph{quantum} relation on a von Neumann algebra $\A \subseteq \B(H)$ to be a weak*-closed operator space $\X \subseteq \B(H)$ that is an $\A'$-bimodule. A \emph{quantum graph} is then a symmetric ($\S = \S^*$) and reflexive ($\A' \subseteq \S$) quantum relation; equivalently, an operator system $\S \subseteq \B(H)$ that is a bimodule over $\A'$. This framework recovers classical and noncommutative examples. When $\A= D_n$, every quantum graph arises from a classical graph $\G$ as a graph operator system $\S_{\G} = \operatorname{span}\{\varepsilon_{i,j} : i \simeq_{\G} j\}.$
When $\A = M_n$, a quantum graph is simply an operator system $\S \subseteq M_n$, corresponding to the noncommutative graphs associated with quantum channels as shown in \cite[Lemma~2]{DUAN09}. We note that in this work we will only deal with quantum graphs acting on finite-dimensional Hilbert spaces.

\subsection{Morita equivalence}
The study of Morita equivalence in the context of noncommutative analysis was introduced by Rieffel \cite{Rief74} in the 1970’s, who reformulated classical Morita theory for rings in the setting of C*- and W*-algebras. Rieffel’s approach proved to be highly effective, as many fundamental results of Morita theory \cite{Bass62} extend to these categories. Morita equivalence, realized as an equivalence between categories of C*-modules (and W*-modules respectively) is strictly weaker than $*$-isomorphisms and under suitable assumptions it coincides with stable isomorphisms \cite{BGR77} (i.e., the stabilizations obtained by tensoring with the compact operators are $*$-isomorphic).

The development of operator space theory paved the way for new applications of Morita theory. Blecher, Muhly and Paulsen \cite{BMP00} defined Morita equivalence of operator algebras, extending Morita equivalence of C*-algebras. Two operator algebras $\A,\B$ are said to be Morita equivalent when there exist appropriate bimodules $\M$ and $\N$ over $\A$ and $\B$ such that
\[
\A\cong \M\otimes_\B \N \qand \B\cong \N\otimes_\A \M.
\]
Following the same direction, Blecher and Kashyap \cite{BK08} defined weak Morita equivalence for dual operator algebras.

During that time, significant attention was drawn to ternary rings of operators (TROs) that is, operator spaces $\X \subseteq \B(H)$ satisfying $\X \X^* \X \subseteq \X$. Such structures arise naturally in the context of Morita equivalence of C*-algebras, as well as in the study of normalisers of reflexive algebras \cite{KT03}.
Using these objects, Eleftherakis introduced TRO-equivalence \cite{Ele12}. A key advantage of TRO-equivalence is that it captures the selfadjoint structure, to the extent that it is present in an operator algebra. Building on this framework, Eleftherakis introduced $\Delta$-equivalence \cite{Ele08}, a Morita type equivalence that generalizes Rieffel’s theory, and which coincides with stable isomorphism in the case of dual operator algebras \cite{EP08}, as well as for operator algebras \cite{Ele16}, under the additional assumption of the existence of a $\sigma$-unit.

Moreover, the nature of $\Delta$-equivalence allowed for further extensions of the theory to operator spaces and operator systems. This was first carried out by Eleftherakis, Paulsen, and Todorov \cite{EPT10}, who established its connection with stable isomorphism for dual operator spaces. In that work, it was shown that $\Delta$-equivalence is well suited to capturing the multiplicative structure carried by operator spaces, as reflected in their multiplier algebras \cite[Theorem 2.13]{EPT10}.

Eleftherakis and Kakariadis \cite{EK17} extended these results to operator spaces by proving that $\Delta$-equivalent operator spaces have TRO-equivalent TRO-envelopes \cite[Theorem 5.10]{EK17}, and in the unital case TRO-equivalent C*-envelopes \cite[Theorem 4.11]{EK17}. In \cite[Theorem 4.3]{EK17}, it is shown that stable isomorphism implies $\Delta$-equivalence. Moreover, under additional natural assumptions such as (approximate) unitality \cite[Corollary 4.7]{EK17} or separability \cite[Corollary 4.8]{EK17}, the two notions coincide.
Subsequently, Eleftherakis, Kakariadis, and Todorov \cite{EKT21} introduced $\Delta$-equivalence for operator systems and showed that $\Delta$-equivalent operator systems admit TRO-equivalent representations inside their C*-envelopes. These representations further restrict to a TRO-equivalence between their corresponding multiplier algebras \cite[Proposition 3.9]{EKT21}. Furthermore, they proved that, in this setting, $\Delta$-equivalence coincides with stable isomorphism.

In addition, Eleftherakis, Kakariadis, and Todorov \cite{EKT25} provide a characterisation of Morita equivalence in the operator system category via a tensorial decomposition, where the analytic structure is induced by symmetrisation. This result establishes an operator system analogue of the factorisation Morita theorem in other categories.
Finally, in recent work of the third-named author \cite{Kou25}, $\Delta$-equivalence and stable isomorphism are shown to be equivalent in the setting of dual operator systems, i.e., operator systems that are also dual operator spaces.

\subsection{Morita equivalence for quantum graphs}
Our goal is to develop a notion of Morita equivalence for quantum graphs and to identify the structural features that are preserved under this equivalence. Links between Morita theory and quantum graph isomorphisms were established by Musto, Reutter and Verdon \cite{MRV19}. In particular, it was shown that graphs quantum isomorphic to a given graph are in bijective correspondence with Morita equivalence classes of certain Frobenius algebras in the category of finite-dimensional representations of the quantum automorphism algebra of that graph. Our motivation comes from \cite[Corollary 7.6]{EKT21}, where Eleftherakis, Kakariadis and Todorov proved that for graphs $\G$ and $\H$,  the following are equivalent:
\begin{enumerate}
\item $\S_{\G} \sim_{\rm TRO} \S_{\H}$;
\item $\S_{\G} \sim_{\Delta} \S_{\H}$;
\item $\G$ and $\H$ are full pullbacks of isomorphic graphs.
\end{enumerate}

A graph $\G$ is said to be a \emph{pullback} of $\H$ if there exists a map
\[
f \colon V(\G)\to V(\H) \quad \text{such that} \quad  x \simeq_{\G} y \iff f(x) \simeq_{\H} f(y).
\]
Such a map is called a \emph{pullback}. If moreover, $f$ is surjective 
it is called a \emph{full} pullback.

We show that the property of being full pullbacks of isomorphic graphs corresponds precisely to being clique blow-ups of isomorphic graphs or equivalently having an isomorphism between the skeletons of the graphs (see Proposition \ref{P:Redgr} and Remark \ref{rmk:cliqueblow}). We define the \emph{skeleton} of a graph $\G$ to be the graph $\G^\circ$ obtained by identifying vertices that are true twins (vertices which have the same closed neighborhoods), and declaring two resulting vertices to be adjacent whenever the original vertices are adjacent. Any graph $\G$ is a full pullback of its skeleton $\G^\circ$. In addition, we show that the multiplier algebra  of a graph operator system admits a canonical decomposition into a direct sum of matrix blocks where each block corresponds to a true-twin equivalence class (see Proposition \ref{P:twin}). 

The main purpose of this work is to  extend the result of \cite{EKT21} to arbitrary quantum graphs.
Our first objective is to quantize the notion of the pullback between graphs (see Definition \ref{D:pullbackmap} and Proposition \ref{prop:classical-weaver-pullback}). The natural candidate for a \emph{pullback between quantum graphs} $(\T,\B)$ and $(\S,\A)$ turns out to be a unital $*$-homomorphism $\theta\colon \B \to \A$ with a Kraus representation  
\[
\theta(x) = \sum_{i=1}^rv_i^*xv_i \quad \text{such that} \quad v_i^*\T v_j \subseteq \S  \qand v_i\S v_j^* \subseteq \T.
\]
A \emph{full pullback} is a pullback that is moreover a faithful $*$-homomorphism. In the case that such a $\theta$ exists we say that \emph{$(\S,\A)$ is a (full) pullback of $(\T,\B)$}.
The first condition can be thought of--- in the classical case--- as preserving adjacency-or-equality, while the second condition can be thought of as reflecting adjacency-or-equality. The fullness condition corresponds to a surjective pullback function in the classical situation. This notion naturally connects with homomorphisms of quantum graphs, as introduced by Stahlke \cite{Sta16} in the setting of trace-free noncommutative graphs, and with the recent work of \cite{KL26}.

Pullbacks and pushforwards were defined in \cite{Wea12} in the more general context of quantum relations.  Let $(\S,\A)$ and $(\T, \B)$ be quantum graphs acting on finite-dimensional Hilbert spaces $H$ and $K$, and let $\vphi\colon \B\to \A$ be a unital completely positive map with a Kraus representation 
\[
\vphi(b) = \sum_{i=1}^rv_i^*bv_i \foral b\in \B.
\]
The \emph{pushforward of $(\S,\A)$ along $\vphi$}, denoted by
$\S^{\rightarrow\vphi}$, is the $\B'$-operator bimodule generated by
\[
\{v_i s v_j^* : s\in\S \text{ and } i,j=1,\dots,r\}
\subseteq \B(K),
\]
and the \emph{pullback of $(\T, \B)$ along $\vphi$}, denoted by
$\T^{\leftarrow\vphi}$, is the $\A'$-operator bimodule generated by 
\[
\{v_i^*t v_j : t\in\T \text{ and } i,j=1,\dots,r \}
\subseteq \B(H).
\]
In \cite{Daws24} and in \cite{Wea12} it is proved that $\S^{\rightarrow\vphi}$ and $\T^{\leftarrow\vphi}$ are independent of the Kraus representation.
In Theorem \ref{thm:ucppullback} we prove that if $\theta$ is a $*$-homomorphism, then $\theta$ is a pullback if and only if $\S=\T^{\leftarrow \theta}$, and moreover $\theta$ is a full pullback if and only if 
$\S=\T^{\leftarrow \theta}$ and $\T=\S^{\rightarrow \theta}$.

Our second objective is to construct the \emph{skeleton of an irreducibly acting quantum graph} as a noncommutative analogue of the skeleton of a graph. We isolate this construction in Subsection~\ref{sec:quantumreduction} as it may be of independent interest, in the hope of contributing to the broader development of noncommutative combinatorics.
This construction retrieves the true-twin reduction of a classical graph (see Proposition~\ref{P:qskgraphs}). In particular, the decomposition of the multiplier algebra isolates multiplicity spaces corresponding to such ``true-twin vertices", and the operator system factors accordingly. The resulting quantum graph $(\R,\C)$ plays the role of the skeleton of $(\S,\A)$, while the original quantum graph is recovered from it via a canonical amplification along these multiplicity spaces. In particular, as in the classical case $(\S,\A)$ is a full pullback of $(\R,\C)$.

\medskip

\noindent
{\bf Theorem A.} (\cite[Theorem 7.5]{EKT21}, Theorem \ref{th:stronghom})
{\it Let $(\S,\A)$ and $(\T, \B)$ be quantum graphs acting irreducibly on finite-dimensional Hilbert spaces $H$ and $K$. The following are equivalent:
\begin{enumerate}
\item $\S \sim_{\rm TRO} \T$.
\item $\S \sim_{\Delta} \T$.
\item There exists a quantum graph $(\R,\C)$ acting on a finite-dimensional Hilbert space such that $(\S, \A)$ and $(\T, \B)$ are full pullbacks of $(\R, \C)$.
\item There exists a C*-algebra $\C \subseteq \B(L)$ on a finite-dimensional Hilbert space $L$ and unital $*$-homomorphisms $\theta\colon \C \to \A$ and $\varrho\colon \C \to \B$ such that
\[
\T^{\rightarrow \varrho} = \S^{\rightarrow \theta}, \quad 
\T = (\S^{\rightarrow \theta})^{\leftarrow \varrho} \qand 
\S = (\T^{\rightarrow \varrho})^{\leftarrow \theta}.
\]
\end{enumerate}
}

\medskip

The irreducibly acting assumption allows us to view the C*-envelope of $\S$, and hence also the multiplier algebra $\A_{\S}$, concretely acting on the same Hilbert space as $\S$, something that plays a central part in our proof. The equivalence [(i) $\Leftrightarrow$ (ii)] was shown in \cite[Theorem 7.5]{EKT21}.
The equivalence [(iii) $\Leftrightarrow$ (iv)] follows from Theorem \ref{thm:ucppullback} where we connect full pullback maps in our sense and the pushforwards and pullbacks of \cite{Wea12}. The implication [(iii) $\Rightarrow$ (i)] relies on Proposition \ref{prop:nondeg_TRO}, and the implication [(i) $\Rightarrow$ (iii)] relies on the skeleton of a quantum graph.

We then proceed to investigate a second, stronger notion of Morita equivalence for quantum graphs via simultaneous TRO-equivalence of the operator systems and their associated algebras. Let $H,K$ be finite-dimensional Hilbert spaces. For C*-algebras $\A\subseteq \B(H)$ and $\B\subseteq \B(K)$ and a unital completely positive map $\vphi \colon \B \to \A$ with Kraus operators $v_i \in \B(H,K), i\in[r]$, we define 
\[
\X_{\vphi}:= \spn\{\B' v_i \A':i\in [r]\}.
\]
In \cite[Proposition 6.1]{KL26} it is proved that $\X_\vphi$ does not depend on the choice of Kraus operators $v_i$ for $i\in [r]$.

\medskip

\noindent
{\bf Theorem B.} (Theorem~\ref{T:qgraphs_balanced})
{ \it Let $(\S,\A)$ and $(\T,\B)$ be quantum graphs acting on finite-dimensional Hilbert spaces
$H$ and $K$, respectively. The following are equivalent:
\begin{enumerate}
\item There is a non-degenerate TRO $\M\subseteq \B(H,K)$ which is also a $\B'$-$\A'$-bimodule implementing the TRO-equivalences
\[
\S \sim_{\rm TRO } \T \qand \A \sim_{\rm TRO } \B.
\]

\item There exists a faithful unital completely positive map $\vphi\colon\B\to \A$ such that
\[
\X_\vphi^*\T \X_\vphi\subseteq \S,\quad \X_\vphi \S \X_\vphi^*\subseteq \T 
\qand
\X_\vphi^*\B \X_\vphi\subseteq \A,\quad \X_\vphi \A \X_\vphi^*\subseteq \B.
\]
\end{enumerate}
}

In the special case of graph operator systems, this stronger notion of Morita equivalence gives graph isomorphism (see Remark \ref{R:graphstrtro}). For noncommutative graphs the two notions of Morita equivalence coincide and the equivalence admits a particularly concrete formulation in terms of completely positive maps. In this setting, TRO-equivalence can be characterised entirely through pullback and pushforward operations associated to a single map.

\medskip

\noindent
{\bf Theorem C.} (Corollary~\ref{T:ncgraphs})
{\it Let $\S \subseteq M_n$ and $\T \subseteq M_m$ be noncommutative graphs.  
The following are equivalent:
\begin{enumerate}
\item $\S \sim_{\rm TRO} \T$.
\item There exists a unital completely positive map $\vphi \colon M_m \to M_n$ such that 
\[
\S = \T^{\leftarrow \vphi} \qand 
\T = \S^{\rightarrow \vphi}.
\]
\end{enumerate}
}

\medskip

A main topic in the study of noncommutative graphs is how to extend classical graph parameters to the operator-algebraic setting, see for example \cite{BTW21,DSW13,GL20,LPT18}. Several basic invariants, including the independence number, clique number, chromatic number, and different capacity-type quantities, have been extended to noncommutative graphs. These noncommutative versions stay closely related to the classical ones while also allowing truly quantum features such as entanglement and nonlocality. It is therefore important to understand how these parameters behave under natural equivalence relations, especially under Morita equivalence, in order to build a strong and flexible theory of quantum graphs.

In Remark \ref{R:chrqli} we show that the chromatic number and the clique number of a noncommutative graph are not invariant under TRO-equivalence.
However, Theorem C makes it possible to compare several noncommutative graph parameters directly under TRO-equivalence.

\medskip

\noindent
{\bf Theorem D.} (Theorem \ref{th:nc_graph_parameter_invariance})
{\it The parameters $\alpha, c_{0}, \vartheta, \wh{\vartheta}, \fH, \beta$ and $\ga$ are invariant under TRO-equivalence of noncommutative graphs. In particular, they are invariant under $\Delta$-equivalence of irreducibly acting noncommutative graphs.
}

\medskip

\subsection{Contents}
The structure of the manuscript is as follows. In Section~\ref{sec:prelim}, we recall the necessary background on operator systems, noncommutative graphs, and Morita-type equivalences, including $\Delta$-equivalence and TRO-equivalence. We also revisit $\De$-equivalence of graph operator systems in terms of the true-twin equivalence relation and clique blow-ups. In Section~\ref{sec:qgraphs}, we introduce our framework for quantum graphs, establish the basic properties of pullback and pushforward constructions. The main part of the section is devoted to the development of Morita type-equivalences for quantum graphs, culminating in Theorem~\ref{th:stronghom} and Theorem~\ref{T:qgraphs_balanced}, which characterises these equivalences in several ways. We also present a true-twin reduction construction of irreducibly acting quantum graphs. This construction plays a central role in the proof Theorem~\ref{th:stronghom} and it may be interesting in its own merit. We specialize to noncommutative graphs and obtain a concrete description in terms of completely positive maps. Finally, in Section~\ref{sec:nonc}, we study invariance results for various noncommutative graph parameters under Morita equivalence.

\section{Preliminaries} \label{sec:prelim}

\subsection{Operator systems and TROs} 
For an excellent introduction to the theory of operator spaces and operator systems, the reader is addressed to \cite{BL04, Pau02}.
In this work, all inner products are assumed to be linear on the second variable. For $n \in \bN$ we write $[n]$ for the set $\{1,\dots,n\}$. We write $M_{n}$ for the $n\times n$ matrices with complex entries and $D_n$ for the diagonals. Furthermore, we use $\eps_{ij}$ to denote the $(i, j)$-th matrix unit in $M_{n}$, for $i, j \in [n]$. We use $\perp$ to denote orthogonality with respect to the Hilbert--Schmidt inner product on $M_n$  defined by
\[
\sca{x, y} = \operatorname{Tr}(x^* y)  \foral x,y \in M_n.
\]

Let $H$ and $K$ be  Hilbert spaces.
For $\xi,\eta$ in $H$, the rank one projection $ \xi\eta^*\colon H \to H$ acts as $ \xi\eta^*(\ze)=\sca{\eta, \ze} \xi$. We write $\B(H, K)$ for the space of all bounded linear operators from $H$ into $K$, and set $\B(H) = \B(H, H)$. We write $I_{H}$ or just $I$ for the identity operator on $H$, when the context is clear. If $\X$ is a set of operators in $\B(H,K)$ we set 
\[
[\X H]:= \ol{\spn}\{x\xi: \; x \in \X, \xi \in H\}.
\]
For a set $\X$ of operators inside $\B(H)$ we will write $\X'$ for their commutant, i.e.,
\[
\X':=\{ a\in \B(H) : a b=ba \foral b \in \X\}.
\]

An \textit{operator space} $\X$ is a norm-closed subspace of $\B(H, K)$. An \textit{operator system} is an operator space $\S \subseteq \B(H)$ with $I_{H} \in \S$ and $\S^{*} = \S$.
An operator space $\M\subseteq \B(H,K)$ is called a \emph{TRO} if it satisfies 
\[
\M\M^*\M\subseteq \M.
\]
It follows directly from the definition that for a TRO $\M$ the norm-closed subspaces $[\M^*\M]\subseteq \B(H)$ and $[\M \M^*] \subseteq \B(K)$ are C*-algebras. We say that  an operator space $\X \subseteq \B(H,K)$ is  \emph{non-degenerate} (resp. \emph{acts non-degenerately}) if $I_H\in [\X^*\X] $ and $I_K\in [\X\X^*] $ (resp. $[\X H]=K$ and $ [\X^*K]=H$).

Let $\X, \Y$ be operator spaces. For a linear map $\phi \colon \X \to \Y$ and $n \in \bN$, the $n$-th \textit{amplification} $\phi^{(n)} \colon M_n(\X) \to M_n(\Y)$ is defined by
\[
\phi^{(n)}((x_{ij})_{i,j}) := (\phi(x_{ij}))_{i,j}.
\]
We say that $\phi$ is \textit{completely bounded} if
\[
\|\phi\|_{\mathrm{cb}} := \sup_{n \in \mathbb{N}} \|\phi^{(n)}\| < \infty.
\]
It is called \textit{completely contractive} if $\|\phi\|_{\mathrm{cb}} \leq 1$, and \textit{completely isometric} if $\phi^{(n)}$ is an isometry for all $n \in \mathbb{N}$.

Now let $\S, \T$ be operator systems. 
We say that $\phi \colon \S \to \T$ is \textit{unital} if $\phi(1_\S)=1_\T$, and \textit{positive} if $\phi(\S_+) \subseteq \T_+$. The map $\phi$ is called \textit{completely positive} if each $\phi^{(n)}$ is positive for all $n \in \mathbb{N}$. We say that $\phi$ is a \textit{complete order isomorphism} if it is a bijection whose inverse is completely positive, and a \textit{complete order embedding} if it is a complete order isomorphism onto its range.

For an operator system $\S$, a \emph{C*-envelope} of $\S$ is a pair $(\cenv(\S), \io)$ consisting of a C*-algebra $\cenv(\S)$ and a unital completely isometric map
$\io \colon \S \to \cenv(\S)$
such that $\ca(\io(\S)) = \cenv(\S)$ and the following universal property holds:
whenever $\vphi \colon \S \to \B(H)$ is a unital complete isometry, there exists a unique surjective $*$-homomorphism
\[
\pi \colon \ca(\vphi(\S)) \to \cenv(\S)
\text{ such that }
\pi \circ \vphi = \io.
\]
Such a pair exists, and it is unique up to $*$-isomorphism. The existence of the C*-envelope of an operator system was established by Hamana \cite{Ham79} through the existence of the injective envelope.
An independent proof for unital operator algebras was established by Dritschel and McCullough \cite{DM05} through the existence of maximal dilations. In particular, they show that the C*-envelope is the C*-algebra generated by a unital completely isometric maximal representation.
A simplified proof by Arveson \cite{Arv08} yields the same result for operator systems. 

An operator system $\S \subseteq \B(H)$ \emph{acts irreducibly} on $H$ if there is no non-trivial closed subspace $L \subseteq H$ 
reducing $\S$ such that the restriction map
\[
\S \to \B(L) ; s \mapsto s|_{L}
\]
is completely isometric. For an irreducibly acting operator system $\S\subseteq M_n$ there exist positive integers $n_{1},\dots,n_{k}$ 
with $n_{1}+\cdots+n_{k} = n$ such that
\[
\cenv(\S) \cong \ca(\S) \cong \bigoplus_{i=1}^{k} M_{n_{i}}.
\]
This is shown by Arveson in \cite{Arv11}. For a simplified proof see also \cite[Proposition 7.1]{EKT21}.

We recall the following facts and definitions from \cite{EKT21}. Let $\S \subseteq \B(H)$ and $\T \subseteq \B(K)$ be operator systems. 
We say that $\S$ and $\T$ are \emph{TRO-equivalent (via $\M$)}, and write 
\[
\S \sim_{\rm TRO} \T,
\]
if there exists a non-degenerate TRO $\M \subseteq \B(H,K)$ such that
\[
\M^* \T \M \subseteq \S
\qand
\M \S \M^* \subseteq \T.
\]

It is clear that in this case $\S $ is an $[\M^*\M]$-bimodule and $ \T$ is an $[\M\M^*]$-bimodule.
Our operator systems are assumed to be norm-closed and therefore as shown in \cite[Proposition 3.3]{EKT21} if $\S$ and $\T$ are TRO-equivalent via $\M \subseteq \B(H,K)$, then the non-degeneracy of $\M$ implies
\[
[\M^* \T \M] = \S
\qand
[\M \S \M^*] = \T.
\]

We say that two operator systems $\S$ and $\T$ are called \emph{$\De$-equivalent},
and write $\S \sim_{\De} \T$ if there exist Hilbert spaces $H$ and $K$ 
and unital complete order embeddings 
\[
\varphi \colon \S \to \B(H)
\qand
\psi\colon \T \to \B(K)
\]
such that 
\[
\varphi(\S) \sim_{\rm TRO} \psi(\T).
\]

For background on multiplier algebras of operator spaces, we refer the reader to \cite{BL04}.
A \emph{left multiplier} of an operator space $\X$ is a map $u \colon \X\to \X$ for which there exist a Hilbert space $H$, a complete isometry $i \colon \X\to \B(H)$ and an operator $c\in \B(H)$ such that
\[
i(u(b))=c  i(b) \foral b\in \X.
\]
The \textit{multiplier norm} of $u$ is defined as the $\inf{\|c\|}$ over all such triples $(i,c,H)$, and the set of all left multipliers of $\X$ is denoted by $\M_\ell(\X)$. Analogous definitions hold for right multipliers $\M_r(\X)$.

The \textit{left adjointable multiplier algebra} of $\X$ is 
\[
\A_\ell(\X):=\M_\ell(\X)\cap \M_\ell(\X)^*,
\]
and it consists of all $u \colon \X\to \X$ such that there exists a triple $(i,c,H)$ as before with 
\[
i(u(b))=ci(b) \foral b\in \X
\qand 
c^*i(\X)\subseteq i(\X).
\]
In case that $\S$ is an operator system, we have 
\[
\M_\ell(\S)\cong \{a\in \cenv(\S):a\S\subseteq \S\},
\]
and therefore
\begin{align*}
\A_\ell(\S) 
&\cong 
\{a\in \cenv(\S):a\S\subseteq\S,a^*\S\subseteq\S\}\\
&=
\{a,b\in \cenv(\S):a\S b\subseteq \S\}\\
&=
\{a\in \cenv(\S):\S a\subseteq\S,\S a^*\subseteq\S\}
\cong 
\A_r(\S).
\end{align*}
In particular, both $\A_\ell(\S)$ and $\A_r(\S)$ can be viewed as an operator subsystem of $\S$, denoted by $\A_\S$, and carrying the structure of a C*-subalgebra of $\cenv(\S)$.

In the sequel we will use this finite-dimensional modification of \cite[Proposition 3.4]{EKT21}. Although the proof is similar we include the details for the convenience of the reader. 
\begin{proposition} \label{prop:nondeg_TRO}
Let $\S \subseteq \B(H)$ and $\T \subseteq \B(K)$ be two operator systems acting on finite-dimensional Hilbert spaces. Suppose that  there exists a non-degenerately acting operator space $\X \subseteq \B(H,K)$ such that 
\[ 
\X \S \X^* \subseteq \T \qand \X^*\T \X \subseteq \S.
\]
Set $\A:= \ca(\X^*\X)$. Then $ \M:=[\X\A]$ is a non-degenerate TRO implementing $\S \sim_{\rm TRO} \T$.
\end{proposition}
\begin{proof}
Since $\X$ acts non-degenerately, we have $[\X H]=K$ and $[\X^* K]=H$.
Hence the C*-algebras $\ca(\X^*\X)$ and $\ca(\X\X^*)$ act non-degenerately,
and by finite-dimensionality they must be unital. We first note that
\[
\M\M^*\M \subseteq [\X\A\X^*\X\A]
\subseteq [\X\A]
=\M,
\]
so $\M$ is a TRO. 

We  show that $[\M^*\M] = \A$ and $[\M\M^*] = \ca(\X\X^*)$. Since $\A$ is unital, we have $\X \subseteq \M$,
and hence $[\X^*\X] \subseteq [\M^*\M]$.
As $[\M^*\M]$ is a C*-algebra, it follows that
\[
\A=\ca(\X^*\X)\subseteq [\M^*\M].
\]
Conversely, for $x,y\in\X$ and $a,b\in\A$ we have
\[
(xa)^*(yb)=a^*x^*yb \in \A,
\]
since $x^*y\in [\X^*\X]\subseteq \A$.
Thus $[\M^*\M]\subseteq \A$, and so $[\M^*\M]=\A$.
A similar argument gives $[\M\M^*]=\ca(\X\X^*)$.
In particular, $\M$ is non-degenerate.

Now we show that $\S$ is an $\A$-bimodule.  Set $\P:=[\X^*\X]$.
Since $I_K\in \T$ and $\X^*\T\X\subseteq \S$, we obtain
$\P\subseteq \S$.
Moreover,
\[
\P \S \P \subseteq [\X^*\X \S \X^*\X]
\subseteq [\X^*(\X\S\X^*)\X]
\subseteq [\X^*\T\X]
\subseteq \S.
\]
By induction, $\P^n\S \P^n\subseteq \S$ 
for all $n \geq 1$.
Since $\A=\ca(\P)$ we conclude that
\[
\A\S \A\subseteq \S,
\]
and by unitality of $\A$ we also have $\A \S\subseteq \S$ and $\S\A\subseteq \S$.
Finally, we have
\[
\M^*\T\M
\subseteq [\A\X^*\T\X\A]
\subseteq [\A\S\A]
\subseteq \S,
\]
and similarly
\[
\M\S\M^*
\subseteq [\X\A\S\A\X^*]
\subseteq [\X\S\X^*]
\subseteq \T.
\]
Thus $\M$ is a non-degenerate TRO implementing
$\S \sim_{\rm TRO} \T$, and the proof is complete.
\end{proof}

The next result follows from the arguments in the proof of \cite[Proposition 3.9]{EKT21}; we include a proof as it plays a central role in the sequel--- see also \cite[Theorem 5.10]{EK17}.

\begin{proposition}\label{P:multeq}
Let $\S\subseteq \B(H)$ and $\T\subseteq \B(K)$  be two operator systems acting irreducibly on finite-dimensional Hilbert spaces. If $\S \sim_{\rm TRO}\T$ then we can pick a non-degenerate TRO $\N\subseteq \B(H,K)$ implementing $\S \sim_{\rm TRO}\T$ such that $[\N^* \N]= \A_\S$ and $[\N\N^*] = \A_\T$.
\end{proposition}

\begin{proof}
Let $\M\subseteq \B(H,K)$ be a non-degenerate TRO such that
\[
[\M^* \T \M] = \S
\qand
[\M \S \M^*] = \T.
\]
Note that by the TRO-property $\S$ and $\T$ are bimodules over $[\M^*\M]$ and $[\M \M^*]$, respectively. As both $\S$ and $\T$ are irreducibly acting we have 
\[
\A_\S \subseteq \ca(\S) \subseteq \B(H) \qand \A_\T \subseteq \ca(\T) \subseteq \B(K).
\]
Hence by definition of the multiplier algebra we obtain
\[
[\M^*\M]\subseteq \A_\S \qand [\M \M^*] \subseteq \A_\T.
\]
Fix $t\in \T$. Then $t= \sum_i c_i' s_i c_i^*$ for $c_i,c_i'\in \M$ and $s_i \in \S$. For any $d , d' \in \M$ and $a\in \A_\S$ we have
\[
d' a d^* t=  \sum_i d' a d^*   c_i' s_i c_i^*\in [\M \A_\S \M^* \M \S \M^*] \subseteq \T,
\]
and thus $[\M \A_\S \M^*] \subseteq \A_\T$ by the definition of $\A_\T$. Similarly, we have $[\M^* \A_\T \M] \subseteq \A_\S$. In particular, we have equalities since 
\[
\A_\S \subseteq [\M^* \M \A_\S \M^* \M] \subseteq [\M^* \A_\T \M] \subseteq \A_\S \qand \A_\T \subseteq  [\M \M^* \A_\T \M \M^*]  \subseteq \A_\T.
\]
Set $\N:=[\A_\T \M \A_\S]$. Then 
\[
[\N \S \N^*]=[\A_\T \M \A_\S \S \A_\S \M^* \A_\T]= [\A_\T \M \S \M^* \A_\T]= [\A_\T \T \A_\T]=\T,
\]
and similarly $[\N^* \T \N] = \S$. Moreover,
\[
[\N \N^*] = [\A_\T \M \A_\S \A_\S \M^* \A_\T] =\A_\T \qand [\N^* \N]=[\A_\S \M^* \A_\T \A_\T \M \A_\S]=\A_\S.
\]
Thus $\N$ is the required non-degenerate TRO implementing a TRO-equivalence between $\S$ and $\T$, and the proof is complete.
\end{proof}

We will also need the following lemma.

\begin{lemma} \label{L:unitary}
Let $\M\subseteq \B(\bC^\alpha\otimes\bC^n,\bC^\beta\otimes \bC^n)$ be a TRO such that 
\[
[\M^*\M]=M_\alpha\otimes I_n \qand [\M\M^*]=M_\beta\otimes I_n.
\]
Then we have  $\M=M_{\beta,\alpha}\otimes u$ for some unitary $u\in M_n$.
\end{lemma}

\begin{proof}
Set $\A:=[\M^*\M]$ and $\B:=[\M\M^*]$. By \cite[Theorem 3.2]{Ele12}, there exists a $*$-isomorphism $\pi \colon \A'\to\B'$ such that
\begin{equation} \label{eq:tro}
\M=\{x\in M_{\beta,\alpha}\otimes M_n: xa'=\pi(a')x \foral a'\in \A'\}.
\end{equation}
Observe that $\A'=I_\alpha\otimes M_n$ and $\B'=I_\beta\otimes M_n$ and thus $\pi$ induces a $*$-automorphism of $M_n$. In particular, we have
\[
\pi(I_\alpha\otimes a)=I_\beta\otimes (uau^*) \foral a\in M_n,
\]
for some unitary $u\in M_n$.

We will prove that $\M= M_{\be, \al}\otimes u$. Note that the inclusion $M_{\beta,\alpha}\otimes u\subseteq \M$ is immediate. Hence is remains to show the reverse inclusion. 
To that end, fix $x=\sum_{i,j}\eps_{ij}\otimes x_{ij}\in \M$. Then \eqref{eq:tro} implies that
\[
x(I_\alpha\otimes a)= \pi(I_\alpha\otimes a) x= (I_\beta\otimes uau^*)x \foral a\in M_n.
\]
By expanding the above equality we obtain
\[
\sum_{i,j}\eps_{ij}\otimes (x_{ij} a)=\sum_{i,j}\eps_{ij}\otimes (uau^* x_{ij}) \foral a \in M_n,
\]
which implies that
\[
x_{ij}a=u a u^* x_{ij} \foral i,j \text{ and } a\in M_n.
\]
We thus obtain that $u^*x_{ij}\in M_n'=\bC I_n$ for all $i,j$. We conclude that $x\in M_{\beta,\alpha}\otimes u$ and the proof is complete.
\end{proof}

\subsection{Noncommutative graphs and graph operator systems}
Let $H$ and $K$ be finite-dimensional Hilbert spaces. A \emph{quantum channel} is a map $\Phi\colon \B(H)\rightarrow \B(K)$ which is completely positive and trace-preserving. By \cite[Theorem 1]{Cho75} we have that for a quantum channel $\Phi$ there exist operators $v_{i}\colon H\rightarrow K$ for $i = 1, \dots, r$ such that 
\begin{gather}\label{eqn_kraus_decomp}
\Phi(c) = \sum\limits_{i=1}^{r}v_{i}cv_{i}^*, \foral c \in \B(H) \qand \sum_{i=1}^{r}v_{i}^*v_{i} = I_{H}.
\end{gather}
The equation (\ref{eqn_kraus_decomp}) is called a \textit{Kraus representation of $\Phi$}, and the operators $(v_{i})_{i=1}^r$  are the \textit{Kraus operators} of the representation. We note that the second equality in (\ref{eqn_kraus_decomp}) is equivalent with the trace-preservation of the map $\Phi$. The subspace
\[
\S_{\Phi} := \spn\{v_{i}^*v_{j}: \; i, j \in [r]\}
\]
is called the \textit{noncommutative graph of} $\Phi$; by trace-preservation, this is an operator system in $\B(H)$.
We also set
\[
\X_{\Phi} := \spn\{v_{i}: \; i \in [r]\} \subseteq \B(H, K)
\]
as the \textit{Kraus space for} $\Phi$. It can be shown (see \cite[Corollary 2.23]{Wat18}) that both $\S_{\Phi}$ and $\X_{\Phi}$ are  independent of the choice of the Kraus representation for $\Phi$.
In fact, by similar arguments, for a general (not necessarily completely positive) trace-preserving linear map $\Phi\colon \B(H)\rightarrow \B(K)$ there exist operators $v_{i}, w_{i}: H\rightarrow K$ for $i = 1, \hdots, r$ such that
\[
\Phi(c) = \sum\limits_{i=1}^{r}v_{i}cw_{i}^{*} \foral c \in \B(H) \text{ and } \sum\limits_{i=1}^{r}w_{i}^{*}v_{i} = I_{H}.
\]
For such a linear map, we write
\[
\wt{\S}_{\Phi} := \spn\{w_{i}^{*}v_{j}: \; i, j \in [r]\}.
\]
The distinction between $\S_{\Phi}$ and $\wt{\S}_{\Phi}$ will be made clear in context. 
As in the discussion directly preceding \cite[Lemma 4.1]{BTW21}, we set
\[
\fC(\S) := \{\Phi\colon \B(H)\rightarrow M_{k}:\text{ where } \Phi \text{ is a quantum channel with } \S_{\Phi} \subseteq \S \text{ and } k \in \bN\}. 
\]
Finally, for a quantum channel $\Phi$ as above we let $\Phi^{*}: \mathcal{B}(K)\rightarrow \mathcal{B}(H)$ denote the dual of $\Phi$; this is a unital, completely positive map which in finite-dimensions has Kraus representation
\[
\Phi^{*}(d) = \sum\limits_{i=1}^{r}v_{i}^{*}dv_{i}, \foral d \in \mathcal{B}(K),
\]
when $\Phi$ has a Kraus representation as in (\ref{eqn_kraus_decomp}).

We now give some connections with classical graph theory. Let $\G=(V(\G), E(\G))$ be a graph. We say a graph $\G$ is \textit{simple} if there are no loops at any vertices in $\G$. For $x, y \in V(\G)$, we write $ x \simeq_{\G} y$ whenever $ x \sim_{\G} y$ or $ x=y$. Throughout this paper, all graphs are finite, simple, and undirected. When convenient, we implicitly identify the vertex set $V(\G)$ with $[n]$ where $|V(\G)|=n$.
The \emph{graph operator system} of $\G$ is defined as
\[
\S_{\G} := \spn \{\eps_{xy} : x \simeq_{\G} y\}\subseteq M_n.
\]
For a graph $\G$ and a vertex $x\in V(\G)$ we set
\[
\N(x):=\{ y\in V(\G) : x\simeq_{\G} y\},
\]
and call this set the \emph{closed neighborhood} of $x$.
Two vertices $x, y \in V(\G)$ are called  \emph{true twins} if $\N(x)=\N(y)$.
We write $x\approx y$ for two vertices that are true twins, and it follows that $\approx$ defines an equivalence relation on $V(\G)$. We denote by $ [x]$ the true-twin equivalence class of $x$.  The \textit{skeleton}   of $\G$ is the graph $\G^\circ$ obtained by replacing the vertices of $\G$ with their true-twin  equivalence classes and where $[x]$ is adjacent to $[y]$ in $\G^\circ$ whenever $x$ is adjacent to $y$ in $\G$.

Let $\G$ be a graph on $n$ vertices. In \cite[Theorem 3.2]{OP14} it is proven that $\cenv(\S_\G)$ is the C*-algebra generated by $\S_\G$ inside $M_n$. In particular, the multiplier algebra $\A_{\S_\G}$ can be realised inside $M_n$ as well.

\begin{proposition}\label{P:twin}
Let $\G$ be a graph on $n$ vertices. Then 
\begin{equation} \label{eq:multdecomp}
\A_{\S_{\G}}=\bigoplus_{j=1}^k M_{n_j} \text{ for } n_j, k\in \bN \text{ and } n_1+\cdots +n_k=n,
\end{equation}
where each $M_{n_j}$ corresponds to the true-twin equivalence class $C_j$ of size $ n_j$ for all $ j =1,\dots,k $.

Moreover, any irreducibly acting operator system $\S \subseteq M_n$ whose multiplier algebra admits a decomposition as in \eqref{eq:multdecomp} is necessarily a graph operator system.

\end{proposition}

\begin{proof}
We begin by proving that
\[
\varepsilon_{xy} \in \A_{\S_\G}
\quad \Longleftrightarrow \quad
x \approx y.
\]
Fix $x, y\in V(\G)$ such that $x \approx y$ and let $\eps_{wz} \in \S_{\G}$. By definition of $\S_{\G}$, we have $z \in \N(w)$.
To show that $\eps_{xy} \in \A_{\S_\G}$, it suffices to verify that
$\eps_{xy}\eps_{wz}$ and $\eps_{wz}\eps_{xy}$ are both in $\S_{\G}$.
We prove only that $\eps_{xy}\eps_{wz} \in \S_{\G}$, since the argument to show that $\eps_{wz}\eps_{xy} \in \S_{\G}$ is analogous.
Indeed, it follows that
\[
\eps_{xy}\eps_{wz}
=
\begin{cases}
\eps_{xz}, & \text{if } y = w, \\
0, & \text{if } y \neq w,
\end{cases}
\]
so it suffices to consider the case $y = w$.
In this case, we obtain $z \in \N(w)\equiv \N(y) = \N(x)$, and hence
$\eps_{xz} \in \S_{\G}$, as required.

Conversely, suppose that $\eps_{xy} \in \A_{\S_{\G}}$ for some $x, y\in V(\G)$,
and let $w \in \N(y)$.
Then $\eps_{yw} \in \S_{\G}$, and therefore
\[
\eps_{xw} = \eps_{xy}\eps_{yw} \in \S_{\G}.
\]
This implies that $w \in \N(x)$, showing that $\N(y) \subseteq \N(x)$.
By symmetry, we conclude that $\N(x) = \N(y)$, and thus $x \approx y$. 

Let $C_1,\dots,C_k$ be the true-twin equivalence classes of $\G$, and set
\[
p_j := \sum_{x\in C_j} \varepsilon_{xx} \foral j=1,\dots,k.
\]
By definition of $\A_{\S_\G}$ we have $\eps_{xx} \in \A_{\S_\G}$ for every $x\in V(\G)$. In particular, we have that each $p_j \in \A_{\S_\G}$, and moreover these projections are mutually orthogonal and satisfy $\sum_{j=1}^k p_j = I$.
From the previous step, we have
$\varepsilon_{xy} \in \A_{\S_\G}$ if and only if $x\approx y$,
so
\[
p_j \A_{\S_\G} p_j
=
\mathrm{span}\{\varepsilon_{xy} : x, y \in C_j\}.
\]
The family $\{\varepsilon_{xy} : x,y \in C_j\}$ forms a full system of
matrix units on $\ell^2(C_j)$, hence
\[
p_j \A_{\S_\G} p_j \cong M_{n_j},
\qquad n_j = |C_j|.
\]
If $j \neq \ell$, then $x \not\approx y$ for $x \in C_j$, $y \in C_\ell$,
so $p_j \A_{\S_\G} p_\ell = \{0\}$. Therefore,
\[
\A_{\S_\G}
=
\bigoplus_{j=1}^k p_j \A_{\S_\G} p_j
\cong
\bigoplus_{j=1}^k M_{n_j},
\]
as required.

Now let  $\S\subseteq M_n$  be an irreducibly acting operator system and   assume $\A_{\S} = \oplus_{j=1}^k M_{n_j}$ is a block-diagonal algebra in the standard basis.
Since $n_1+\dots+ n_k=n$ we have that $\A_\S$ contains all diagonal matrix units $\eps_{ii}$ and thus $D_n \subseteq \A_{\S}$. By definition of the multiplier algebra $\S$ is a $D_n$-bimodule. By \cite[Proposition 1.1]{PPS89}  it follows that $\S=\S_{\G}$ for some graph $\G$ on $n$ vertices. 
\end{proof}

While the following result is standard, we include a proof for completeness.

\begin{proposition}\label{P:irrgra}
Let $\G$ be a graph on $n$ vertices. The graph operator system $\S_\G\subseteq M_n$ is irreducibly acting. 
\end{proposition}

\begin{proof}
There exist $k\in \bN$ and $n_1,\dots, n_k$ such that $n_1+\dots + n_k=n$ and
\[
\ca(\S_\G)= \bigoplus_{i=1}^k M_{n_i},
\]
see for example the proof of \cite[Theorem 3.2]{OP14}. Suppose that $L\subseteq H$ is reducing for $\S_\G$ (and hence also for $\ca(\S_{\G})$), and the map
\[
s\in \S_\G \mapsto s|_L\in \B(L)
\]
is completely isometric.
In particular, we obtain that
\[
L=\bigoplus_{j\in F} \bC^{n_j} \text{ for some } F\subseteq [k].
\]
Assuming that $L$ is non-trivial, we may pick $i\in [k]$ such that $i\not \in F$. We then get 
\[
\eps_{ii}\in \S_\G\qand \eps_{ii}|_L=0,
\]
thus deriving a contradiction.
\end{proof}

Let $\G = (V(\G), E(\G)) $ and $\H = (V(\H), E(\H)) $ be graphs.  A \emph{graph homomorphism} from  $\G$  to  $\H$  is a map $f \colon V(\G) \to V(\H)$
satisfying
\[
x \sim_{\G} y \quad \Rightarrow\quad f(x) \sim_{\H} f(y) \foral x,y \in V(\G).
\]
If moreover $f$ is bijective and $f^{-1}$ is a graph homomorphism $f$ will be called \emph{graph isomorphism} and we will say that the graphs $\G$ and $\H$ are \emph{isomorphic}.
 
A \emph{pullback} is a map $f\colon V(\G) \to V(\H)$ satisfying
\[
x\simeq_{\G} y \quad \Leftrightarrow \quad f(x) \simeq_{\H} f(y)
 \foral x,y \in V(\G).
\]
If moreover $f$ is surjective we say that $ f$ is a \emph{full} pullback. In this case we say that $\G$ is a (full) pullback of $\H$.

In \cite[Corollary 7.6]{EKT21} it is shown that for graphs $\G$ and $\H$, the statements $\S_{\G} \sim_{\De} \S_{\H}$, $\S_{\G} \sim_{\rm TRO} \S_{\H}$, and $\G$ and $\H$ being full pullbacks of isomorphic graphs are all equivalent. In the following proposition we prove that these statements are also equivalent to $\G^\circ$ and $\H^\circ$ being isomorphic. 

\begin{proposition}\label{P:Redgr}
Let $\G$ and $\H$ be graphs. The following are equivalent:
\begin{enumerate}
\item $\G^\circ$ and $\H^\circ$ are isomorphic.
\item $\G$ and $\H$ are full pullbacks of isomorphic graphs.
\end{enumerate}
\end{proposition}
\begin{proof}
\noindent
[(i) $\Rightarrow$ (ii)]. Define
\[
f_\G \colon V(\G) \to V(\G^\circ) \text{ by } f_\G(x)=[x] \qand f_\H \colon V(\H) \to V(\H^\circ) \text{ by } f_\H(z)=[z].
\]
Note that by definition of the skeletons 
\[
x\simeq_\G y \Leftrightarrow f_\G(x)\simeq_{\G^\circ} f_\G(y) \qand w\simeq_\H z \Leftrightarrow f_\H(w)\simeq_{\H^\circ} f_\G(z),
\]
and therefore $\G$ and $\H$ are full pullbacks of $\G^\circ$ and $\H^\circ$, respectively.

\medskip
\noindent
[(ii) $\Rightarrow$ (i)]. Suppose that $\K$ is a graph and let $f \colon V(\G) \to V(\K)$ be a full pullback map. Consider the skeleton $ \G^\circ$ of $\G$; we will show that $f$ induces a graph isomorphism from $ \G^\circ$ to $\K^\circ$. 

First of all, define the map $ f^\circ \colon \G^\circ \to \K^\circ$ by $ f^\circ([x])= [f(x)]$. To see that $f^\circ$ is well-defined and injective note that $[x] \neq [y]$  if and only if there exists a $z \in V(\G)$ such that $x \simeq_{\G} z$ while $y \not\simeq_{\G} z$ (or vice versa). 
As $f$ is a full pullback this is equivalent to the existence of a $w \in V(\K)$ such that 
$f(x) \simeq_{\K} w$ while  $f(y) \not\simeq_{\K} w$ (or vice versa), which is equivalent to $[f(x)] \neq [f(y)]$.
Moreover, $f^\circ$ is clearly surjective, as $f$ is surjective.

Finally, if $ [x]\sim_{\G^\circ} [y]$ then $x \sim_{\G}y$ and there exists a $z \in \G$ such that $x \sim_{\G} z$ but $ z \nsim_{\G} y$ (or vice versa). This means that $ f(x) \simeq_{\K} f(z)$ and $ f(z) \not \simeq_{\K} f(y) $ which implies $ [f(x)] \sim_{\K^\circ} [f(y)]$. Reversing the arguments yields $ [x] \sim_{\G^\circ} [y]$ if and only if $ [f(x)] \sim_{\K^\circ} [f(y)]$. In particular, $f^\circ$ is a graph isomorphism. Hence if $f \colon V(\G) \to V(\K)$ and $ g\colon V(\H) \to V(\K) $ are full pullback maps we conclude that $\G^\circ \cong \K^\circ \cong \H^\circ$, as required.
\end{proof}

Let $\G=(V(\G),E(\G))$ be a graph and let $\bo{n}=(n_x)_x\in \bN^{V(\G)}$.
The \emph{$\bo{n}$-clique blow-up} of $\G$, denoted by $\G\langle \bo{n} \rangle$, is the graph obtained by replacing each vertex $x\in V(\G)$ with a clique $C_x$ of size $n_x$, and declaring that for $x\neq y$, every vertex of $C_x$ is adjacent to every vertex of $C_y$ if and only if $x$ is adjacent to $y$ in $\G$.

\begin{remark} \label{rmk:cliqueblow}
Let $\G$ and $\H$ be graphs. Then $\G$ and $\H$ are clique blow-ups of isomorphic graphs if and only if their skeletons are isomorphic.
Indeed, if 
\[
\G=\K\langle \bo{n}\rangle \qand \H=\F\langle \bo{m}\rangle
\]
for isomorphic graphs $\K\cong \F$, then by the definition of adjacency of blow-ups of graphs we obtain 
\[
\G^\circ=\K^\circ \cong \F^\circ=\H^\circ.
\]
Conversely, if $\G$ and $\H$ have isomorphic skeletons, then each is obtained as a clique blow-up of its skeletons by replacing each vertex with its corresponding true-twin equivalence class.
\end{remark}

\begin{figure}[!ht]
\centering
\begin{tikzpicture}


\filldraw[black] (-2, 1) circle (2pt) node[anchor=east] {$a_1$};
\filldraw[black] (-2,-1) circle (2pt) node[anchor=east] {$a_2$};
\filldraw[black] (0, 0) circle (2pt) node[anchor=south] {$a_{3}$};

\filldraw[black] (2, 2) circle (2pt) node[anchor=west] {$a_4$};
\filldraw[black] (2, 0) circle (2pt) node[anchor=west] {$a_5$};
\filldraw[black] (2,-2) circle (2pt) node[anchor=west] {$a_6$};

\draw (-2, 1)--(-2,-1);

\draw (-2, 1)--(0,0);
\draw (-2,-1)--(0,0);

\draw (0,0)--(2,2);
\draw (0,0)--(2,0);
\draw (0,0)--(2,-2);

\draw (2,2)--(2,0);
\draw (2,2)--(2,-2);
\draw (2,0)--(2,-2);

\node at (0,-3) {$\G = P_3\langle (2,1,3)\rangle$};


\filldraw[black] (6,0) circle (2pt) node[anchor=south] {$b_{1}$};

\filldraw[black] (8,1) circle (2pt) node[anchor=south] {$b_2$};
\filldraw[black] (8,-1) circle (2pt) node[anchor=north] {$b_3$};

\filldraw[black] (10,1) circle (2pt) node[anchor=south] {$b_4$};
\filldraw[black] (10,-1) circle (2pt) node[anchor=north] {$b_{5}$};

\draw (6,0)--(8,1);
\draw (6,0)--(8,-1);

\draw (8,1)--(8,-1);

\draw (8,1)--(10,1);
\draw (8,1)--(10,-1);
\draw (8,-1)--(10,1);
\draw (8,-1)--(10,-1);

\draw (10,1)--(10,-1);

\node at (8,-3) {$\H = P_3\langle (1,2,2)\rangle$};

\end{tikzpicture}
\caption{An example of graphs with TRO-equivalent graph operator systems. Both are clique blow-ups of the path $P_3$ and their skeleton is $P_3$.}
\label{fig:blowups}
\end{figure}
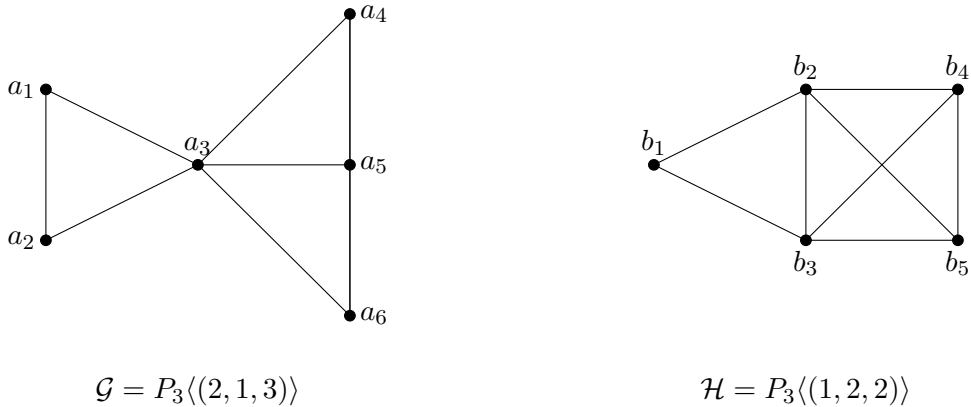


\section{Morita equivalence for quantum graphs}\label{sec:qgraphs}

In this section, we pursue the following main objectives. First, we introduce the notions of a co-homomorphism and an associated pullback map in the setting of quantum graphs, and examine their relationship to the pullback constructions defined in \cite{Wea21} and \cite{EKT21}. Second, we construct the skeleton of an irreducibly acting quantum graph and show that it extends the notion of a skeleton of a graph. Third, 
we extend \cite[Corollary 7.6]{EKT21} from the classical graph setting to the broader framework of irreducibly acting quantum graphs. Finally, we characterise a simultaneous TRO-equivalence of quantum graphs and their associated algebras and apply it in the setting of noncommutative graphs.

\subsection{Co-homomorphisms and pullbacks of quantum graphs}
We begin by recalling the notion of a quantum graph, as introduced in \cite{Wea12}. Let $H$ be a finite-dimensional Hilbert space, and $\A \subseteq \B(H)$ a von Neumann algebra, i.e., $\A=\A''$. A \textit{quantum graph} on $\A$ is an operator system $\S \subseteq \B(H)$ that is a bimodule over $\A'$, that is,  $\A'\S\A' \subseteq \S$. For brevity, if $\S$, $\A$ and $H$ are as above we will often say that $(\S,\A)$ is a quantum graph acting on $H$. We may assume without loss of generality that for a quantum graph $(\S,\A)$ the algebra $\A$ is unital and $1_\S=1_\A=I_H$. We will say that $(\S,\A)$ is \emph{acting irreducibly} on $H$ if $\S$ is acting irreducibly on $H$.

\begin{example}
Let $\G$ be a graph on $n$ vertices. Then the corresponding graph operator system $\S_{\G} \subseteq M_{n}$ is a quantum graph on $\B = D_{n}$. In particular, by \cite[Proposition 1.1]{PPS89} we have that an operator subsystem $ \S\subseteq M_n$ that is a bimodule over $D_n$, gives rise to a classical graph in the sense that there exists a graph $ \G$ on $n$ vertices such that $\S=\S_{\G}$.
\end{example}

\begin{example}
Any operator system $\S \subseteq \B(H)$ on a finite-dimensional Hilbert space $H$ is trivially a bimodule over $\bC I_H$, and hence it defines a quantum graph on $\A = \B(H)$. Such quantum graphs will be called \emph{noncommutative graphs}. In \cite[Lemma~2]{DUAN09} it is shown that every operator system $\S \subseteq \B(H)$ arises as the noncommutative graph of a quantum channel, so this terminology is consistent.
\end{example}

We recall the definitions of pushforwards and pullbacks of quantum graphs, as given in \cite{Wea21}. Let $(\S,\A)$ and $(\T, \B)$ be quantum graphs acting on Hilbert spaces $H$ and $K$, respectively.
Let $\vphi\colon \B\to \A$ be a unital completely positive map with Kraus form
\[
\vphi(b)=\sum_{i=1}^r v_i^* b v_i,
\qquad v_i\in \B(H,K).
\]
\begin{enumerate}
\item
The \emph{pushforward of $(\S, \A)$ along $\vphi$}, denoted by
$\S^{\rightarrow\vphi}$, is the $\B'$-operator bimodule generated by
\[
\{v_i s v_j^* : s\in\S \text{ and } i,j=1,\dots,r\}
\subseteq \B(K).
\]
\item
The \emph{pullback of $(\T, \B)$ along $\vphi$}, denoted by
$\T^{\leftarrow\vphi}$, is the $\A'$-operator bimodule generated by
\[
\{v_i^*t v_j : t\in\T \text{ and } i,j=1,\dots,r \}
\subseteq \B(H).
\]
\end{enumerate}

\begin{remark}
It is shown in \cite{Wea21} (see also \cite{Daws24}) that $\S^{\rightarrow\vphi}$ and $\T^{\leftarrow\vphi}$ are independent of the Kraus representation of $\vphi$. Moreover, in \cite[Theorem 7.5]{Daws24} it is shown that the weak*-closure of the linear span of the set 
\[
\{v_i^*t v_j : t\in\T \text{ and } i,j \in [r] \}
\]
is automatically an $\A'$-bimodule and hence in our finite-dimensional setting we have 
\[
\T^{\leftarrow\vphi}=\spn\{v_i^*t v_j : t\in\T \text{ and } i,j=1,\dots,r \}.
\]
\end{remark}

The following definition of a co-homomorphism is the quantum graph analogue of the one given in \cite[Section 7]{EKT21}.
\begin{definition}
Let $(\S,\A)$ and $(\T, \B)$ be quantum graphs acting on  Hilbert spaces $H$ and $K$, respectively.
A \emph{co-homomorphism} (from $\T$ to $\S$) is a unital completely positive map
\[
\vphi \colon \B \longrightarrow \A
\]
with Kraus operators $(v_i)_{i=1}^r$ such that
\[
\vphi(b)=\sum_{i=1}^r v_i^*b v_i
\qand
v_i^*\T v_j \subseteq \S
\foral i,j=1,\dots,r.
\]
\end{definition}

\begin{remark}
The notion of a co-homomorphism \cite{EKT21} coincides with that of a noncommutative graph homomorphism as introduced by Stahlke in \cite{Sta16}. In \cite{Sta16} the theory is formulated in terms of orthogonal complements of operator systems rather than the operator systems themselves. Since our approach focuses on operator systems, it is more natural to work directly with the present formulation. In the classical setting, this correspondence recovers the usual notion of graph homomorphism. Indeed one may verify that by~\cite[Theorem~8]{Sta16}, for graphs $\G$ and $\H$, there exists a graph homomorphism from $\G$ to $\H$ if and only if there exists a co-homomorphism from $\S_{\ol{\H}}$ to $\S_{\ol{\G}}$ 
where by $\ol{\G}$ and $\ol{\H}$ we denote the complement graphs of the graphs $\G$ and $\H$, respectively.
\end{remark}

Let $(\S,\A)$ and $(\T, \B)$ be quantum graphs acting on  Hilbert spaces $H$ and $K$, respectively, and let $ \vphi \colon \B \to \A$ be a completely positive map. We may assume  that $\vphi$ admits some Kraus operators $(v_i)_{i=1}^r\subseteq \B(H,K)$ (for example, use Arveson's Extension Theorem to extend $\vphi$ to $\B(K)$) such that 
\[
\vphi(b) = \sum_{i=1}^rv_i^* b v_i \foral b\in \B.
\]
Define 
\[
\X_{\vphi}:= \spn\{bv_ia: b\in \B', a \in \A', i=1,\dots,r\},
\]
that is, the $\B'$-$\A'$ bimodule generated by $ \spn\{v_i:i\in [r]\}$.
By \cite[Proposition 6.1]{KL26} we have that $\X_\vphi$ does not depend on the choice of Kraus operators $v_i$ for $i \in [r]$.

\begin{definition} \label{D:pullbackmap}
Let $(\S,\A)$ and $(\T, \B)$ be quantum graphs acting on  Hilbert spaces $H$ and $K$, respectively. 
\begin{enumerate}
\item A  \emph{strong co-homomorphism} from $ (\T,\B)$ to $ (\S,\A)$  is a unital completely positive map $\vphi \colon \B \to \A$ with Kraus operators $(v_i)_{i=1}^r$ such that 
\[
v_i^*\T v_i \subseteq \S, \qand
v_i\,\S\,v_j^* \subseteq \T
\foral i,j=1,\dots,r.
\]
\item A \emph{pullback} from $ (\T,\B)$ to $ (\S,\A)$ is a unital $*$-homomorphism $\theta : \B \to \A$  that is a strong co-homomorphism.
\end{enumerate}
If in addition, $\vphi$ is faithful (resp. $\theta$)
we say that $\vphi$  is a \emph{full} strong co-homomorphism (resp. $\theta$ is a  full pullback). 
If $\theta$ is a (full) pullback we say that \emph{$(\S,\A)$ is a (full) pullback of $(\T,\B)$}.
\end{definition}
Clearly for a unital completely positive map
\[
\text{pullback} \Rightarrow \text{ strong co-homomorphism} \Rightarrow \text{co-homomorphism}.
\]
We note that if $ \vphi \colon \B \to \A$ is a strong co-homomorphism from $ (\T,\B)$ to $ (\S,\A)$, then using the bimodule structures of $\S$ and $\T$ yields
\[
\X_{\vphi}^*\T \X_{\vphi} \subseteq \S \qand  \X_{\vphi} \S \X_{\vphi}^* \subseteq \T.
\]
On the other hand, if $ \X_{\vphi}^*\T \X_{\vphi} \subseteq \S$ and $ \X_{\vphi} \S \X_{\vphi}^* \subseteq \T$ then  by \cite[Proposition 6.1]{KL26} we may pick Kraus operators $(v_i)_{i=1}^r$ for $\vphi$ so that $\X_{\vphi} = \spn\{v_i:i\in [r]\}$ and thus obtain the inclusions
\[
v_i^*\T v_j \subseteq \S \qand v_i\,\S\,v_j^* \subseteq \T
\foral i,j=1,\dots,r.
\]
We will make regular use of \cite[Proposition 6.1]{KL26} throughout this paper without further mention.

\medskip

Let $\B \subseteq \B(K)$ be a von Neumann algebra and $ v\in \B(K)$. Set $p_\B(v)$ to be the $\B$-support projection of $v$, that is,
\[
p_\B(v):=\inf\{p \in {\rm proj}(\B) : \ran(v) \subseteq \ran(p)\}.
\]
\begin{lemma}\label{L:injectivity}
Let $(\S,\A)$ and $(\T, \B)$ be quantum graphs acting on Hilbert spaces $H$ and $K$, respectively. Let $\vphi \colon \B \to \A$ be unital completely positive with Kraus operators $(v_i)_{i=1}^r$ that span $ \X_{\vphi}$ and set $v=\sum_{i=1}^rv_iv_i^*$. Then, the following hold:
\begin{enumerate}
\item The $\B$-support projection  $p_{\B}(v)$ is the projection onto $ [\X_{\vphi}H]$. In particular, 
\[ 
[\X_\vphi H] = K \quad \text{if and only if} \quad  p_\B(v)=I_K;
\]
\item The map $\vphi $ is faithful if and only if $ [\X_{\vphi}H]=K$;

\end{enumerate}
\end{lemma}
\begin{proof}
\noindent 
(i) We need to show that $p_\B(v)=q$ where $q$ is the projection on $[\X_\vphi H]$. We start by showing that $p_\B(v)\leq q$. It is evident that $[\X_\vphi H]$ is invariant for $\B'$ and hence it is reducing. Consequently, we obtain $q \in \B''= \B$. It remains to show that $\ran(v) \subseteq\ran (q)$ and then by the definition of $p_\B(v)$ we will get $p_\B(v) \leq q$. But this follows from the fact that $[\X_\vphi H] \supseteq \ran(v)$ as both $\A'$ and $\B'$ are unital. We now show that $q\leq p_\B(v)$. Note that we have
\[
(I_K-p_\B(v)) ( \sum_{i=1}^rv_iv_i^*)(I_K -p_\B(v))=0,
\]
and therefore we obtain that 
\[
(I_K-p_\B(v)) v_i =0 \foral i \in [r],
\]
and
\[
p_\B(v)v_i=v_i \foral i \in [r].
\]
In particular, $[\X_\vphi H] \subseteq \ran (p_\B(v))$ as required.

\medskip
\noindent
(ii) Suppose that $\vphi$ is faithful. We have 
\[
(I_K-p_\B(v))v_i=0 \foral i\in[r]
\]
and thus
\[
\vphi((I_K-p_\B(v))^*(I_K-p_\B(v)))=\vphi((I_K-p_\B(v))) = \sum_iv_i^*(I_K-p_\B(v))v_i=0.
\]
Faithfulness of $\vphi$ implies that $I_K-p_\B(v)=0$.

Conversely, assume that $[\X_{\vphi}H]=K$. For $ x^*x\in \ker\vphi$ we have
\[
0=\vphi(x^*x)= \sum_{i=1}^rv_i^*x^*xv_i = \sum_{i=1}^r(xv_i)^* (xv_i),
\]
which implies that $ xv_i=0$ for all $i\in [r]$. Consequently $ x v_i \eta=0$ for all $i\in[r]$ and all $\eta \in H$. Since $ [\X_{\vphi}H]=K$ we obtain that $x \xi=0$ for all $ \xi \in K$ implying that $x=0$. Thus $\vphi$ is faithful. 
\end{proof}

\begin{remark}\label{rmk_mult}
Let  $ \theta\colon \B \to \A $ be a unital $*$-homomorphism between finite dimensional von Neumann algebras with $ \theta (x)= \sum_{i=1}^r v_i^*x v_i$.  Then we have
\[
v_iv_j^*\in \B' \foral i,j \in [r].
\]
Indeed, define the column isometry 
\( 
V:=(v_i)_{i=1}^r \in \B(H, K^{(r)}),
\)
and note that
\[
\theta(b)=V^*(b\otimes I_r)V \foral b\in \B.
\]
Since $\theta$ is multiplicative, the range of $V H$ is invariant for $\pi(\B)$, where 
\[
\pi \colon \B \to M_r(\B); \pi(b) = b \otimes I_r.
\]
Hence, we get
\[
V V^{*} \in  \pi(\B)' = \B' \otimes M_r.
\]
and thus 
\[
v_iv_j^*\in \B' \foral i,j \in [r].
\]
\end{remark}

We now provide the connection between co-homomorphism and pullback maps with the pushforwards and pullbacks defined in \cite{Wea21}. 
\begin{theorem}\label{thm:ucppullback}
Let $(\S,\A)$ and $(\T,\B)$ be quantum graphs acting on Hilbert spaces $H$ and $K$, respectively. Let $\vphi \colon \B \to \A$
be a unital completely positive map.
The following are equivalent:
\begin{enumerate}
\item $\S=\T^{\leftarrow \vphi}$ and $\T=\S^{\rightarrow \vphi}$;
\item $\vphi$ is a full strong co-homomorphism.
\end{enumerate}

Moreover, if  $\theta \colon \B \to \A$ is a $*$-homomorphism then $\S=\T^{\leftarrow \theta}$ if and only if $\theta$ is a pullback.
\end{theorem}
\begin{proof}
Choose Kraus operators  such that  $\X_{\vphi} = \spn\{v_i:i\in[r]\}$.

\medskip
\noindent
[(i) $\Rightarrow$ (ii)].
Assume $\S=\T^{\leftarrow \vphi}$ and  $\T=\S^{\rightarrow \vphi}$. The inclusion $v_i^*\,\T\,v_j\subseteq \S$ for all $i,j \in [r]$ follows directly from the definition of $\T^{\leftarrow \vphi}$ and hence $ \X_{\vphi}^* \T \X_{\vphi} \subseteq \S$. On the other hand, we have $ \X_{\vphi} \S \X_{\vphi}^* \subseteq \S^{\rightarrow \vphi} =\T$ and thus $\vphi$ is a strong co-homomorphism.  To see that $\vphi$ is full note that $ I_K\in \T = \S^{\rightarrow\vphi}$ hence 
\[
I_K = \sum_{\ell} b_\ell v_{i(\ell)}s_{\ell} v_{j(\ell)}^*b'_\ell \quad \text{for some } b_\ell, b'_\ell \in \B' \text{ and } s_\ell \in \S.
\]
Assume $\xi \in K$, then 
\[
\xi=I_K \xi = \sum_{\ell} b_\ell v_{i(\ell)} s_{\ell} v_{j(\ell)}^*b_{\ell}'\xi \in [\B'\X_{\vphi} H]= [\X_{\vphi}H].
\]
Thus $[\X_{\vphi} H]= K$ and by Lemma \ref{L:injectivity} we obtain that $\vphi$ is faithful.

\medskip

\noindent
[(ii) $\Rightarrow$ (i)].
Assume that $v_i^*\T v_j\subseteq\S$ and
$v_i\S v_j^*\subseteq\T$ for all $i,j$.
The first inclusion together with the fact that $\S$ is an $\A'$-bimodule implies $\T^{\leftarrow \vphi}\subseteq\S$.
Conversely, let $s$ be in $\S$. Since $\sum_i v_i^*v_i=I_H$, we have
\[
s
=(\sum_i v_i^*v_i) s(\sum_j v_j^*v_j)
=\sum_{i,j} v_i^* (v_i s v_j^*) v_j.
\]
By assumption, $v_i s v_j^*\in\T$, and hence each summand lies in $\T^{\leftarrow \vphi}$.
Therefore $s\in\T^{\leftarrow \vphi}$ and so $\S\subseteq\T^{\leftarrow \vphi}$, as required.

Assume that moreover $ \vphi $ is full, that is $[\X_{\vphi}H]=K$. Since $v_i \S v_j^*\subseteq \T$ and $ \T$ is a $\B'$-bimodule we have $\S^{\rightarrow\vphi}\subseteq\T$.
Using that $ \vphi$ is unital, we can see that $ [\X_{\vphi}^*K]=H$, that is, $\X_{\vphi}$ acts non-degenerately. Now set $ \A_{\X_{\vphi}} =\ca(\X_{\vphi}^*\X_{\vphi})$ and $\M = [ \X_{\vphi} \A_{\X_{\vphi}}]$. Note that $\A_{\X_{\vphi}}$ is unital, $\S$ is an $\A_{\X_{\vphi}}$-bimodule and $\M$ is non-degenerate (see the proof of Proposition \ref{prop:nondeg_TRO}). In particular, $ [\M\M^*] $ is unital and
\[
\M^*\T \M \subseteq \S.
\]
Thus
\[
\T \subseteq [\M\M^* \T \M \M^*] \subseteq [\M\S\M^*] \subseteq [\X_{\vphi}\A_{\X_{\vphi}}\S \A_{\X_{\vphi}} \X_{\vphi}^*]
 \subseteq [\X_{\vphi}\S \X_{\vphi}^*] \subseteq\S^{\rightarrow\vphi},
\]
concluding that $ \T = \S^{\rightarrow\vphi}$.

We now prove the last statement, let $\theta$ be a unital $*$-homomorphism and pick Kraus operators $(v_i)_{i=1}^r$ such that $\X_{\theta} = \spn\{v_i: i\in [r]\}$. The fact that if $\theta$ is a pullback implies that $\S=\T^{\leftarrow \theta}$ is already shown in the first paragraph of the proof of implication [(ii) $\Rightarrow$ (i)]. For the converse assume $\S=\T^{\leftarrow \theta}$.
The inclusion $v_i^*\,\T\,v_j\subseteq \S$ for all $i,j \in [r]$ follows directly from the definition of $\T^{\leftarrow \theta}$.
For the reverse inclusion, fix $s\in\S$. Since $H$ is finite-dimensional, by \cite[Theorem 7.5]{Daws24} we may write
\[
s=\sum_{k,\ell} v_k^*\,t_{k\ell}\,v_\ell
\qquad\text{with } t_{k\ell}\in\T.
\]
Then for fixed $i,j$ we obtain
\[
v_i s v_j^*
=\sum_{k,\ell} (v_iv_k^*)\,t_{k\ell}\,(v_\ell v_j^*).
\]
Since $\T$ is a $\B'$-bimodule and $v_iv_k^*,v_\ell v_j^*\in \B'$ for every $i,j,k,\ell \in [r]$ (see Remark \ref{rmk_mult}) we obtain that
each term of the sum above lies in $\T$ and hence $v_i s v_j^*\in\T$, as required.  
\end{proof}

The following proposition shows that pullbacks of quantum graphs coincide in the classical case with pullbacks of graphs in the sense of \cite{EKT21}.

\begin{proposition}\label{prop:classical-weaver-pullback}
Let $\G$ and $\H$ be graphs.
\begin{enumerate}
\item
If $f \colon V(\G)\to V(\H)$ is a pullback then the map $\theta_f \colon D_{V(\H)} \to D_{V(\G)}$ defined by
\[
\theta_f(\sum_{w\in V(\H)}\la_{w} \eps_{ww})
= \sum _{i \in V(\G)} \lambda_{f(i)} \eps_{ii}
\]
is a unital $*$-homomorphism that satisfies $\S_\G=\S_\H^{\leftarrow\theta_f}$.

\item Conversely, let $\theta \colon D_{V(\H)}\to D_{V(\G)}$ be a unital $*$-homomorphism such that $\S_\G=\S_\H^{\leftarrow\theta}$.
Then there exists a unique pullback $f\colon V(\G)\to V(\H)$ such that $\theta=\theta_f$. 
\end{enumerate}

In addition, $f$ is surjective if and only if $\theta_f$ is injective if and only if $[\X_{\theta_f} \,\bC^{V(\G)} ]= \bC^{V(\H)}$.
\end{proposition}

\begin{proof}
(i) The fact that $\theta_f$ is a unital $*$-homomorphism follows directly from its definition. 
For each $i\in V(\G)$ set
\[
v_i:=\eps_{f(i) i}\in M_{V(\H),V(\G)},
\]
and note that for every $w,z \in V(\H)$ we have
\[
v_i^*\eps_{wz}v_j
=\eps_{i,f(i)}\,\eps_{wz}\,\eps_{f(j),j}
=\de_{f(i),w}\de_{z,f(j)}\,\eps_{ij}.
\]
In particular,
\[
\sum_{i\in V(\G)} v_i^* \eps_{ww} v_i
=
\sum_{i\in V(\G)} \de_{f(i),w}\varepsilon_{ii}=\theta_f(\eps_{ww}),
\]
and thus $\theta_f$ admits the Kraus representation
\[
\theta_f(d)=\sum_{i\in V(\G)} v_i^* d v_i
\quad \text{for all } d \in D_{V(\H)}.
\]
By definition, $\S_\H^{\leftarrow\theta_f}$ is the $D_{V(\G)}$-bimodule generated by
\[
\{v_i^* t v_j: i,j\in V(\G), t\in \S_\H\}.
\]
Since $\S_\H$ is spanned by $\{\eps_{wz}: w\simeq_\H z\}$, the computation above yields
\begin{equation}\label{eq:pullback-span-canonical}
\S_\H^{\leftarrow\theta_f}
=
\spn\{v_i^* \eps_{wz} v_j : i,j\in V(\G), w\simeq_\H z\}
=
\spn\{\eps_{ij}: f(i)\simeq_\H f(j)\}.
\end{equation}
Hence
\[
\S_\G
=
\spn\{\eps_{ij}: i\simeq_\G j\}
=
\spn\{\eps_{ij}: f(i)\simeq_\H f(j)\}
=
\S_\H^{\leftarrow\theta_f},
\]
since $f$ is a pullback.

\medskip
\noindent
(ii) Let $\theta \colon D_{V(\H)}\to D_{V(\G)}$ be a unital $*$-homomorphism.
The projections 
\[
\{\theta(\eps_{ww})\}_{w\in V(\H)}\subseteq D_{V(\G)}
\]
are pairwise orthogonal and sum to $I_{V(\G)}$, so there exists a unique partition
\[
V(\G)=\bigsqcup_{w\in V(\H)} F_w
\text{ such that }
\theta(\eps_{ww})
=
\sum_{i\in F_w} \eps_{ii}.
\]
Define $f \colon V(\G)\to V(\H)$ by $f(i)=w$ if and only if $i\in F_w$. 
Then $\theta=\theta_f$, and uniqueness of $f$ follows from the definition of $\theta_f$.

Assume $\S_\G=\S_\H^{\leftarrow\theta}$. As noted in \cite{Daws24}, the pullback $\S_\H^{\leftarrow\theta_f}$ is independent of the Kraus representation, so we may compute it using the Kraus family $\{v_i:=\eps_{f(i),i}\}_{i\in V(\G)}$. By \eqref{eq:pullback-span-canonical},
\[
\S_\H^{\leftarrow\theta}
=
\S_\H^{\leftarrow\theta_f}
=
\spn\{\eps_{ij}: f(i)\simeq_\H f(j)\}.
\]
Thus
\[
\spn\{\eps_{ij}: i\simeq_\G j\}
=
\spn\{\eps_{ij}: f(i)\simeq_\H f(j)\}.
\]
Linear independence of the matrix units gives
\[
i\simeq_\G j \iff f(i)\simeq_\H f(j)
\quad \forall\, i,j\in V(\G),
\]
so $f$ is a pullback.

Finally, $f$ is surjective if and only if each $w\in V(\H)$ occurs as $f(i)$ for some $i$, which is equivalent to injectivity of $\theta_f$. The equivalence with the fullness condition $[\X_{\theta_f}\,\bC^{V(\G)}]=\bC^{V(\H)}$ follows from Lemma \ref{L:injectivity}.
\end{proof}

\begin{example}
The previous proposition shows that classical full pullbacks are precisely those implemented by unital $*$-homomorphisms. This cannot, in general, be extended to faithful unital completely positive maps.

Indeed: let $\G$ be the empty graph on three vertices and $\K_3$ the complete graph on three vertices. Thus
\[
\S_{\G}=D_3, \qquad \S_{\K_3}=M_3.
\]
Define a unital completely positive map
\[
\vphi \colon D_{V(\G)}\to D_{V(\K_3)}
\]
by
\[
\vphi(\varepsilon_{aa})=\tfrac12(\varepsilon_{11}+\varepsilon_{33}),\qquad
\vphi(\varepsilon_{bb})=\tfrac12(\varepsilon_{11}+\varepsilon_{22}),\qquad
\vphi(\varepsilon_{cc})=\tfrac12(\varepsilon_{22}+\varepsilon_{33}).
\]
Then $\vphi$ is faithful and admits a Kraus representation
\[
\vphi(d)=\sum_k v_k^* d v_k
\]
with Kraus operators
\[
\tfrac1{\sqrt2}\varepsilon_{a,1},\;
\tfrac1{\sqrt2}\varepsilon_{b,1},\;
\tfrac1{\sqrt2}\varepsilon_{b,2},\;
\tfrac1{\sqrt2}\varepsilon_{c,2},\;
\tfrac1{\sqrt2}\varepsilon_{a,3},\;
\tfrac1{\sqrt2}\varepsilon_{c,3}.
\]
A direct computation shows that
\[
\S_{\G}^{\leftarrow\vphi}=M_3=\S_{\K_3}.
\]

However, there is no surjective pullback map $f \colon V(\K_3)\to V(\G)$.
Indeed, the pullback condition
\[
i\simeq_{\K_3}j \iff f(i)\simeq_{\G}f(j)
\]
reduces here to $f(i)=f(j)$ for all $i,j$, so every pullback map must be constant.
Hence the only pullback maps are the three constant maps onto $a$, $b$, or $c$,
none of which is surjective.
\end{example}

\subsection{A true-twin construction for quantum graphs} \label{sec:quantumreduction}
In this subsection we will describe a  construction that will be particularly useful in the main theorem. This can be thought of as the quantum analogue of the true-twin reduction in classical graphs in which we obtain the skeleton of an irreducibly acting quantum graph.

Let $(\S,\A)$ be a quantum graph acting irreducibly on a Hilbert space $H$. Since $\S$ acts irreducibly, we may view $\A_{\S}$ concretely as a unital C*-subalgebra of $\B(H)$. Hence there exist positive integers
\[
\alpha_1,\dots,\alpha_k \qand n_1,\dots,n_k
\]
such that, up to unitary equivalence,
\[
H=\bigoplus_{i=1}^k (\bC^{\alpha_i}\otimes \bC^{n_i})
\qand
\A_{\S}
=
\bigoplus_{i=1}^k (M_{\alpha_i}\otimes I_{n_i})
\subseteq \B(H).
\]
Write
\[
H_i:=\bC^{\alpha_i}\otimes \bC^{n_i}
\foral i\in [k].
\]
Let $p_i $ denote the corresponding minimal central projections of $\A_{\S}$.
We will frequently use the identification 
\[
p_i \B(H) p_j \cong \B(H_j, H_i) \cong M_{\alpha_i,\alpha_j}\otimes M_{n_i,n_j},
\]
thus we view 
\[
p_i=I_{\alpha_i}\otimes I_{n_i}\in \B(H_i)\subseteq \B(H) \foral i\in [k].
\]
For $i,j\in [k]$ set
\[
\S_{ij}:=p_i\S p_j \subseteq \B(H_j,H_i).
\]
Note that $\S_{ij}\subseteq \S$ since $p_i,p_j \in \A_\S$ for all $i,j\in [k]$
and $\S$ is an $\A_{\S}$-bimodule. Moreover, we have
\[
(M_{\alpha_i}\otimes I_{n_i} ) \; \S_{ij } \subseteq \S_{ij} \qand \S_{ij}\;(M_{\alpha_j}\otimes I_{n_j}) \subseteq \S_{ij}.
\]
Indeed, if
\[
a\in M_{\alpha_i}\otimes I_{n_i},\quad b\in M_{\alpha_j}\otimes I_{n_j} \qand x\in \S_{ij},
\]
then write $ x = p_isp_j$, for $ s\in \S$ and hence
\[
axb
=
p_i (a p_i) s (p_j b) p_j
\in p_i \A_\S \S  \A_\S p_j \subseteq 
\S_{ij}.
\]
Fix $i,j\in [k]$ and let $\eps^{(i,j)}_{rs}$ denote the $(r,s)$-matrix unit of
$M_{\alpha_i,\alpha_j}$.
Let $ \eta \in \bC^{\al_i}$ and $ \xi\in \bC^{\al_j}$ be of norm one and consider the rank one operator $\xi\eta^*\in M_{\al_j,\al_i}$.
We set
\[
L_{\xi\eta^*} \colon M_{\alpha_i,\alpha_j}\otimes M_{n_i,n_j}\to M_{n_i,n_j}
\]
to be the left slice map associated with the rank one operator $\xi\eta^*\in M_{\alpha_j,\alpha_i}$, that is,
\[
L_{\xi\eta^*}(a\otimes b)=\operatorname{Tr}((\xi \eta^*) a) b \text{ for } a\otimes b \in  M_{\alpha_i,\alpha_j}\otimes M_{n_i,n_j}.
\]
We define
\[
\R_{ij}:=L_{\eps_{11}^{(j,i)}}(\S_{ij})\subseteq M_{n_i,n_j}.
\]

\begin{lemma}\label{lem:rij_structure}
For each $i,j\in [k]$, let $\R_{ij}$ be defined via the slice construction described above. Then the following hold:
\begin{enumerate}
\item The space $\R_{ij}$ is independent of the choice of rank one operator used in its definition.
\item We have the decomposition
\[
\S_{ij}
=
M_{\alpha_i,\alpha_j}\otimes \R_{ij}.
\]
\end{enumerate}
\end{lemma}

\begin{proof}
(i) To see why $\R_{ij}$ is independent of the choice of rank one operators, we will show that for any rank one operator $ \xi \eta^* \in M_{\al_j,\al_i}$ where $ \eta \in \bC^{\al_i}$ and $ \xi\in \bC^{\al_j}$ are of norm one, we have 
\[
L_{\xi\eta^*}(\S_{ij})= L_{\eps_{11}^{(j,i)}}(\S_{ij}).
\]
Indeed, choose unitaries $ u^i \in M_{\al_i}$ and $u^j \in M_{\al_j}$ with 
\[
u^i (e_1 )=\eta \qand u^j (e_1)=\xi.
\]
Then, simple calculations show that
\[
L_{\xi\eta^*}(x)= L_{\eps_{11}^{(j,i)}}((u^i\otimes I)^*x(u^j\otimes I)) \foral x\in  M_{\alpha_i,\alpha_j}\otimes M_{n_i,n_j}.
\]
Since $\S_{i,j}$ is invariant under left and right multiplication with $ M_{\al_i}\otimes I_{n_i}$ and $M_{\al_j}\otimes I_{n_j}$, respectively, and $u^i, u^j$ are unitaries we conclude that 
\[
L_{\xi \eta^*}(\S_{ij})= L_{\eps_{11}^{(j,i)}}(((u^i)^*\otimes I_{n_i} )\S_{ij}(u^j \otimes I_{n_j}) )=L_{\eps_{11}^{(j,i)}}(\S_{ij}).
\]
\medskip

\noindent
(ii)
For fixed $k,\ell $ and $x=\sum_{k',\ell'}\eps^{(i,j)}_{k'\ell'}\otimes x_{k' \ell'}\in \S_{ij}$ we have
\[
(\eps^{(i,i)}_{1k}\otimes I_{n_i})x(\eps^{(j,j)}_{\ell 1}\otimes I_{n_j})
=\eps^{(i,j)}_{11}\otimes x_{k \ell}\in \S_{ij}.
\]
Applying $L_{\eps_{11}^{(j,i)}}$ yields
\[
x_{k \ell}=L_{\eps_{11}^{(j,i)}}(\eps^{(i,j)}_{11}\otimes x_{k \ell})
\in\R_{ij},
\]
and hence
\[
\S_{ij}\subseteq M_{\alpha_i,\alpha_j}\otimes\R_{ij}.
\]
Conversely, if $\eps^{(i,j)}_{k \ell} \otimes y\in M_{\al_i,\al_j} \otimes \R_{ij}$ choose $x\in \S_{ij}$ with $L_{\eps_{11}^{(j,i)}}(x)=y$.
Then
\[
\eps^{(i,j)}_{11}\otimes y
=
\eps^{(i,j)}_{11}\otimes L_{\eps_{11}^{(j,i)}}(x)
=
(\eps^{(i,i)}_{11}\otimes I_{n_i})x(\eps^{(j,j)}_{11}\otimes I_{n_j}) \in\S_{ij}.
\]
In particular,
\[
\eps^{(i,j)}_{k \ell}\otimes y=(\eps^{(i,i)}_{k1}\otimes I_{n_i})(\eps^{(i,j)}_{11}\otimes y)(\eps^{(j,j)}_{1 \ell}\otimes I_{n_j}) \in \S_{ij}.
\]
We conclude that 
$\S_{ij}
=
M_{\alpha_i,\alpha_j}\otimes\R_{ij}$, as required.
\end{proof}

With the notation above, set
\[
L:=\bigoplus_{i=1}^k \bC^{n_i},
\qquad
\C:=\bigoplus_{i=1}^k M_{n_i}\subseteq \B(L),
\qquad
\R:=\bigoplus_{i,j=1}^k \R_{ij}\subseteq \B(L).
\]
We will call $(\R,\C)$ the \emph{skeleton} of the quantum graph $(\S,\A)$.

\begin{proposition}\label{prop:reduced_qgraph}
The skeleton $(\R,\C)$ of an irreducibly acting quantum graph $(\S,\A)$ is a quantum graph and $(\S,\A)$ is a full pullback of $(\R,\C)$.
\end{proposition}

\begin{proof}
By Lemma~\ref{lem:rij_structure}, we have
\[
\S_{ij}=M_{\alpha_i,\alpha_j}\otimes \R_{ij}
\qquad \text{for all } i,j\in [k].
\]
We first show that $(\R,\C)$ is a quantum graph. 
We have that $\R$ is a selfadjoint subspace since $(\S_{ij})^*=\S_{ji}$, and it
contains $I_L$ as 
\[
L_{\eps_{11}^{(i,i)}}(I_{\alpha_i}\otimes I_{n_i})=I_{n_i}.
\]
Moreover, $\R$ is trivially a bimodule over $\C'=\bigoplus_i\mathbb C I_{n_i}$ and therefore
$(\R,\C)$ is a quantum graph. We now construct the pullbacks. Define $\theta \colon \C \to \A_{\S}'$ by
\[
\theta (\oplus_{i=1}^k x_i)
=
\oplus_{i=1}^k (I_{\alpha_i}\otimes x_i),
\]
and since $\S$ is an $\A'$-bimodule, by the definition of $\A_\S$ we obtain $ \A'\subseteq \A_{\S}$ and hence $\theta(\C)\subseteq \A$.
This is a unital $*$-homomorphism since it is the direct sum of the canonical
amplifications $x_i\mapsto I_{\alpha_i}\otimes x_i$ on each block. 

We will now obtain a Kraus representation for $\theta$. For each $i\in [k]$ and $r\in[\al_i]$ define the operators $ v_{i,r}: H \to L$ by 
\[
v_{i,r}:\begin{pmatrix}
\zeta_1\\
\vdots\\
\zeta_i\\
\vdots\\
\zeta_k
\end{pmatrix} \mapsto \begin{pmatrix}
0 \\
\vdots \\
(e_r^* \otimes I_{n_i})(\zeta_i)\\
\vdots\\
0
\end{pmatrix}
\]
with zeros in all but the $i$-th coordinate and $e_r^*(\xi)=\langle e_r,\xi \rangle$ for $ \xi \in \bC^{\alpha_i}$.   Thus  $v_{i,r} \colon H\to L$ vanishes off $H_i$ and  
\[
v_{i,r}(\sum_s e_s\otimes\xi_s)=\xi_r \quad  \text{for} \quad \sum_s e_s\otimes\xi_s\in H_i.
\]
In addition, we have
\[
v_{i,r}^* : \begin{pmatrix}
\xi_1\\
\vdots\\
\xi_i\\
\vdots\\
\xi_k
\end{pmatrix} \mapsto \begin{pmatrix}
0 \\
\vdots\\
e_r \otimes \xi_i\\
\vdots\\
0
\end{pmatrix}.
\]
Note that
\[
\sum_{r=1}^{\alpha_i} v_{i,r}^*v_{i,r}=p_i
\qand
\sum_{r=1}^{\alpha_i} v_{i,r}v_{i,r}^*=\alpha_i I_{L_i},
\]
and therefore
\[
\sum_{i,r} v_{i,r}^*v_{i,r}=I_H
\qand
\sum_{i,r} v_{i,r}v_{i,r}^*=\bigoplus_{i=1}^k \alpha_i I_{L_i}.
\]
Fix $x=\oplus_i x_i\in\C\subseteq \B(L)$.
For 
\[
\zeta= (\zeta_j)_{j=1}^k\in H \text{ with }  \zeta_j=\sum_{s}e_s\otimes\xi_s^{(j)}\in H_i,
\]
we have
\[
(v_{i,r}^*xv_{i,r})(\zeta)= v_{i,r}^* 
x_i \xi_r^{(i)} = 
e_r \otimes x_i \xi_r^{(i)},
\]
and hence
\[
\sum_{r=1}^{\alpha_i} (v_{i,r}^*xv_{i,r})(\zeta)
=
\sum_{r=1}^{\alpha_i} e_r\otimes x_i\xi_r^{(i)}
=
(I_{\alpha_i}\otimes x_i)\zeta.
\]
Summing over all $i\in [k]$ yields the Kraus form
\[
\theta_{}(x)=\sum_{i=1}^k\sum_{r=1}^{\alpha_i} v_{i,r}^*\,x\,v_{i,r}.
\]

To show the pullback relations, fix $i,j \in [k]$ and  let $y$ be $\R_{ij}$. We claim that
\[
v_{i,r}^*\,y\,v_{j,r'}=(\eps^{(i,j)}_{rr'}\otimes y),
\]
under the usual identification of $\B(H_j,H_i) \subseteq \B(H)$.
To verify this, it suffices to evaluate both sides on elementary tensors $e_s\otimes\eta\in H_j$
where $1\le s\le \alpha_j$ and $\eta\in L_j$. Indeed, we have
\[
v_{i,r}^*\,y\,v_{j,r'}(e_s\otimes\eta)
=
v_{i,r}^*(\delta_{r's}\,y\eta)
=
\delta_{r's}\,(e_r\otimes y\eta),
\]
and
\[
(\eps^{(i,j)}_{rr'}\otimes y)(e_s\otimes\eta)
=
(\eps^{(i,j)}_{rr'}e_s)\otimes y\eta
=
\delta_{r's}\,e_r\otimes y\eta.
\]
Since $y\in\R_{ij}$, we obtain
\[
v_{i,r}^*\,y\,v_{j,r'}=\eps^{(i,j)}_{rr'}\otimes y\in
M_{\alpha_i,\alpha_j}\otimes \R_{ij}=\S_{ij},
\]
showing that 
\[
v_{i,r}^*\R_{ij} v_{j,r'} \subseteq \S_{ij}.
\]
Conversely, if $x\in\S_{ij}$ write
\[
x=\sum_{\mu=1}^{\alpha_i}\sum_{\nu=1}^{\alpha_j} \eps^{(i,j)}_{\mu \nu}\otimes y_{\mu \nu}
\text{ for } y_{\mu\nu}\in\R_{ij}.
\]
Then for $\eta\in L_j$ we obtain
\[
v_{i,r}x v_{j,r'}^*(\eta)
=
v_{i,r}(x(e_{r'}\otimes\eta))
=
v_{i,r}(\sum_{\mu=1}^{\alpha_i} e_\mu\otimes y_{\mu r'}\eta)
=
y_{rr'}\eta,
\]
so
\[
v_{i,r} x v_{j,r'}^*=y_{rr'}\in\R_{ij}.
\]
Therefore, we have
\[
v_{i,r}^*\,\R\,v_{j,r'}\subseteq \S 
\qand
v_{i,r}\,\S\,v_{j,r'}^*\subseteq \R,
\]
and $(\S,\A)$ is a  pullback of $(\R,\C)$.
Finally $\theta$ is clearly injective and hence it is a full pullback.
\end{proof}

The following proposition shows that the skeleton of a graph operator system is the graph operator system of the skeleton of the graph, showing that the construction of the skeleton of a quantum graph is indeed a true-twin reduction analogue in the quantum graph setting.

\begin{proposition}\label{P:qskgraphs}
Let $\G$ be a graph on $n$ vertices and let $k$ be the number of vertices of its skeleton $\G^\circ$. Then the skeleton of the quantum graph $(\S_\G, D_n)$ is the quantum graph $(\S_{\G^\circ}, D_k)$.
\end{proposition}

\begin{proof}
By Proposition \ref{P:irrgra} we have that $\S_\G$ is irreducibly acting. Let $(\R,\C)$ be the skeleton of $(\S_\G,D_n)$. By Proposition \ref{P:twin} and the construction of $(\R,\C)$ we obtain
\[
\A_{\S_{\G}}
=\bigoplus_{i=1}^k M_{\alpha_i}, 
\quad 
L=\bigoplus_{i=1}^k \mathbb C
\qand 
\C=\bigoplus_{i=1}^k M_1\cong D_k \subseteq \B(L),
\]
where the number of blocks $k=V(\G^\circ)$ is exactly the number of true-twin equivalence classes $C_1,\dots,C_k$ of $\G$ and $\al_i= |C_i|$ for all $i\in[k]$. Moreover, we have 
\[
p_i=I_{\al_i}=\sum_{x\in C_i} \eps_{xx} \foral i\in[k],
\]
and since vertices within a true-twin equivalence class have identical neighborhoods
\[
(\S_\G)_{ij}
=
p_i \S_\G p_j
=
\begin{cases}
M_{\alpha_i,\alpha_j} & \text{if } C_i\sim_{\G^\circ}C_j,\\
(0) & \text{otherwise.}
\end{cases}
\]
In particular, we get
\[
\R_{ij}
=
L_{\varepsilon^{(j,i)}_{11}}(\S_{ij})
=
\begin{cases}
\bC & \text{if } C_i\sim_{\G^\circ}C_j,\\
(0) & \text{otherwise,}
\end{cases}
\]
and therefore
\[
\R
=
\bigoplus_{i,j}\R_{ij}
\subseteq \B(L)\cong M_k,
\]
is precisely the operator system spanned by the matrix units $\eps_{ij}$
corresponding to adjacency in the skeleton graph $\G^\circ$.
Thus $\R=\S_{\G^\circ}$.
\end{proof}

\subsection{TRO-equivalence via pullbacks}
In this subsection we state and prove one of the main result of this paper. We also present examples of TRO-equivalent quantum graphs and demonstrate how our result extends \cite[Corollary~7.6]{EKT21} from the classical graph setting to the quantum graph framework.

\begin{theorem} \label{th:stronghom}
Let $(\S,\A)$ and $(\T, \B)$ be two quantum graphs acting irreducibly on Hilbert spaces $H$ and $K$, respectively. The following are equivalent:
\begin{enumerate}
\item $\S \sim_{\rm TRO} \T$.
\item $\S \sim_{\De} \T$.
\item There exists a quantum graph $(\R,\C)$ so that  $(\S, \A)$ and $(\T, \B)$ are full pullbacks of $\R$.
\item There exists a C*-algebra $\C\subseteq\B(L)$ on a finite-dimensional Hilbert space $L$ and unital $*$-homomorphisms $ \theta\colon\C\to \A $ and $ \varrho\colon \C\to \B$ such that 
\[
\T^{ \rightarrow\varrho}= \S^{\rightarrow \theta}, \quad \T= (\S^{\rightarrow \theta})^{\leftarrow \varrho} \qand \S= (\T^{\rightarrow \varrho})^{\leftarrow\theta}.
\]
\end{enumerate}
\end{theorem}
\begin{proof}
\noindent
[(i) $\Leftrightarrow $ (ii)]. It follows from \cite[Theorem 7.5]{EKT21}.

\medskip
\noindent
[(iii) $\Leftrightarrow$ (iv)]. It follows from Theorem \ref{thm:ucppullback}.

\medskip 
\noindent[(iii) $\Rightarrow$ (i)]. Suppose $\theta\colon \C\rightarrow \A$ and $\varrho \colon \C\rightarrow \B$ are unital $*$-homomorphisms with Kraus operators $\{v_{i}\}_{i=1}^{r} \subseteq \B(H, L)$ and $\{w_{j}\}_{j=1}^{s} \subseteq \B(K, L)$ implementing the full pullbacks of $\S$ and $\T$ with $\R$, respectively. We may choose Kraus operators such that
\[
\X_{\theta} = \spn\{v_{i}:  i = 1, \hdots, r\} \subseteq \B(H, L) \text{ and } \X_{\varrho} = \spn\{w_{j}:\; j = 1, \hdots, s\} \subseteq \B(K, L).
\]
Using the fact that both pullbacks are full, Lemma \ref{L:injectivity} implies that
\[
[\X_{\theta} H] = L = [\X_{\varrho} K].
\]
Set 
\[
\X := [\X_{\varrho}^{*} \X_{\theta}] = \spn\{w_{j}^{*}v_{i}:  i \in [r], j \in [s]\} \subseteq \B(H, K).
\]
As $\theta$ and $\varrho$ are unital we obtain 
\[
[\X_{\theta}^{*} L] \supseteq [\X_{\theta}^{*}\X_{\theta} H] = H \qand [\X_\varrho^{*}L] \supseteq [\X_\varrho^{*}\X_\varrho K] = K,
\]
and therefore 
\[
[\X^{*}K] = [\X_{\theta}^{*} \X_{\varrho} K] = [\X_{\theta}^{*} L] = H
\qand
[\X H] = [\X_{\varrho}^{*}\X_{\theta} H] = [\X_{\varrho}^{*}L] = K,
\]
showing that $\X$ acts non-degenerately. 
We have
\begin{align*}
\X\S\X^{*}= \X_\varrho^* \X_\theta \S \X_\theta^* \X_\varrho 
&\subseteq 
\spn\{  w_i^* v_j  \S  v_k^*  w_\ell  : j,k \in [r] \text{ and } i,\ell \in[s]\}\\
&\subseteq
\spn\{  w_i^*  \R  w_\ell  : i,\ell \in[s]\}\subseteq \T.
\end{align*}
A similar argument shows that
\[
\X^{*}\T\X \subseteq \S.
\]
Applying Proposition \ref{prop:nondeg_TRO} yields that $\S \sim_{\rm TRO} \T$, as required. 

\medskip
\noindent
[(i) $\Rightarrow$ (iii)].
Suppose that $\S$ and $\T$ are TRO-equivalent. Since both $\S$ and $\T$ act irreducibly, by Proposition \ref{P:multeq} we may assume that
\[
\S=[\M^* \T\M] \qand \T=[\M \S \M^*],
\]
via a non-degenerate TRO $\M$ satisfying
\[
\A_{\S}=[\M^*\M]
\qand
\A_{\T}=[\M\M^*].
\]
By \cite[Theorem 3.2]{Ele12} we have 
\[
\A_{\S}'\cong \A_{\T}',
\]
and therefore up to unitary equivalence we may write
\[
\A_{\S}
=
\bigoplus_{i=1}^k (M_{\alpha_i}\otimes I_{n_i})
\subseteq \B(H)
\qand
\A_{\T}
=
\bigoplus_{i=1}^k (M_{\beta_i}\otimes I_{n_i})
\subseteq \B(K),
\]
with
\[
H=\bigoplus_{i=1}^k (\bC^{\alpha_i}\otimes \bC^{n_i})
\qand
K=\bigoplus_{i=1}^k (\bC^{\beta_i}\otimes \bC^{n_i}).
\]
Let $(\R^{\S},\C)$ and $(\R^{\T},\C)$ be the skeletons of the quantum graphs $(\S,\A)$ and $(\T,\B)$. Note that $(\R^{\S},\C)$ and $(\R^{\T},\C)$ are quantum graphs on the same algebra due to the decompositions of $H$ and $K$. By Proposition~\ref{prop:reduced_qgraph} we have that $(\S,\A_{}^{})$ and $(\T,\B)$ are full pullbacks of $(\R^{\S},\C)$ and $(\R^{\T},\C)$ via the maps $\theta'$ and $\varrho$, respectively. We will prove that $\R^{\S}$ and $\R^{\T}$ are unitarily equivalent through a unitary conjugation that preserves $\C$, and from that we will derive that $(\S,\A_{}^{})$ and $(\T,\B)$ are full pullbacks of $(\R^{\T},\C)$.

Let $p_i$ and $q_i$ be the minimal central projections corresponding to the above block decompositions. Since the $*$-isomorphism $\A_{\T}'\cong \A_{\S}'$
implemented by \cite[Theorem 3.2]{Ele12} preserves minimal central projections, and hence after permuting blocks we may assume that
\[
q_i x= x p_i \foral x\in \M \text{ and } i=1,\dots,k.    
\]
Therefore, we obtain
\[
\T_{ij}=q_i \T q_j=q_i[\M\S\M^*]q_j
= 
[q_i \M p_i \S p_j\M^* q_j] = [(q_i \M p_i) \S_{ij}(p_j\M^* q_j)].
\]
Moreover, we have
\[
[(q_i\M p_i)^*(q_i\M p_i)]
=[p_i\M^*\M p_i]
=[p_i\A_{\S}p_i]
=M_{\alpha_i}\otimes I_{n_i},
\]
and similarly
\[
[(q_i\M p_i)(q_i\M p_i)^*]
=
[q_i \A_\T q_i]
=
M_{\beta_i}\otimes I_{n_i}.
\]
In particular, each $q_i\M p_i$ is a  TRO implementing Morita equivalence between the C*-algebras
$M_{\alpha_i} \otimes I_{n_i}$ and $M_{\beta_i}\otimes I_{n_i}$. By Lemma \ref{L:unitary} there exists a unitary $u_i \in M_{n_i}$ such that
\[
q_i\M p_i
=
M_{\beta_i,\alpha_i}\otimes u_i.
\]
Hence, we obtain
\begin{align*}
\T_{ij}&= 
\big[(q_i\M p_i)\,\S_{ij}\,(p_j\M^*q_j)\big]\\
&= [(M_{\beta_i,\alpha_i}\otimes u_{i})(M_{\alpha_i,\alpha_j} \otimes \R_{ij}^{\S})(M_{\alpha_j,\beta_j}\otimes u_{j}^*) ]\\
&=M_{\beta_i,\beta_j} \otimes (u_i\R_{ij}^{\S}u_j^*), 
\end{align*}
which implies that
\[
M_{\beta_i,\beta_j}\otimes\R_{ij}^{\T}=\T_{ij}
=M_{\beta_i,\beta_j}\otimes (u_i\R_{ij}^{\S}u_j^*),
\]
and therefore
\[u_i\R_{ij}^{\S}u_j^*
=\R_{ij}^{\T}.
\]
By seting $u= \oplus_{i=1}^k u_i\in \bigoplus_{i=1}^k M_{n_i}$ we get
\[
u\R^{\S}u^*=u(\bigoplus_{i,j}\R_{ij}^{\S})u^*= \bigoplus_{i,j} (u_i\R_{ij}^{\T}u_j^*)
=\bigoplus_{i,j}\R^{\T}_{ij}= \R^{\T}.
\] 
Moreover, we have 
\[
u \C u^*= \bigoplus_i u_i M_{n_i} u_i^*= \bigoplus_i  M_{n_i}= \C.
\]
In particular, $(\R^{\S},\C)$ is a full pullback of $(\R^{\T},\C)$ through the map
\[
\ad{u^*} \colon \C \to \C ; x\mapsto u^*xu.
\]
We conclude that  $(\S,\A)$ is a full pullback of $(\R^{\T},\C)$ through the map $\theta:=\theta'\circ\ad{u^*}$ and $(\T,\B)$ is a full pullback of $(\R^{\T},\C)$ through the map $\varrho$, and the proof is complete.
\end{proof}

\begin{remark}
In \cite[Theorem 2.7]{Wea12} it is shown that the notion of a quantum relation is independent of the Hilbert space on which the underlying von Neumann algebra acts. Some of the results of \cite{Wea12} are revisited in \cite{Daws24} by developing alternative methods using Stinespring dilations.

Let $(\S,\A)$ be a quantum graph acting on a Hilbert space $H$ and let $\theta \colon \A \to \B(K)$ be a faithful unital $*$-representation on a finite-dimensional Hilbert space $K$. Write $\theta(x)=\sum_{i=1}^r v_i^*xv_i$, and let $V$ be the corresponding column isometry. Set
\[
\S_\theta:=V^*(\S\otimes M_r)V.
\]
Then, by \cite[Theorem 7.5]{Daws24}, we have $\S_\theta=\S^{\leftarrow\theta}$,
and $\S_\theta$ is a $\theta(\A)'$-bimodule. Hence, by Theorem~\ref{th:stronghom} and Theorem~\ref{thm:ucppullback}, the quantum graphs $\S$ and $\S_\theta$ are TRO-equivalent. In particular, the correspondence 
\[
\S\mapsto \S_\theta
\]
of \cite[Lemma 7.4]{Daws24} preserves and reflects TRO-equivalence, and thus induces a well-defined map on TRO-equivalence classes.

In particular, if $\pi\colon \A \to \B(H\otimes L)$ is given by $\pi(x)=x\otimes I_L$, then by \cite[Lemma 7.3]{Daws24} the correspondence $\S \mapsto\S \otimes \B(L)$ is a bijection between quantum graphs on $\A$ and on $\pi(\A)$, and since $\S \sim_{\mathrm{TRO}} \S\otimes \B(L)$, it follows that this correspondence  induces a bijection on TRO-equivalence classes.
\end{remark}

\begin{example}
Consider
\[
\C := M_2 \oplus M_3 \subseteq \B(\bC^2 \oplus \bC^3)
\qand
\R :=
\begin{pmatrix}
\bC I_2 & M_{2,3}\\[2mm]
M_{3,2} & \bC I_3
\end{pmatrix}
\subseteq \B(\bC^2 \oplus \bC^3).
\]
Define also
\[
\A := (I_2\otimes M_2)\oplus M_3
\subseteq \B((\bC^2\otimes \bC^2)\oplus \bC^3)
\qand
\S:=
\begin{pmatrix}
M_2\otimes I_2 & M_{2,1}\otimes M_{2,3}\\[2mm]
M_{1,2}\otimes M_{3,2} & \bC I_3
\end{pmatrix},
\]
and
\[
\B := (I_3\otimes M_2)\oplus (I_2\otimes M_3)
\subseteq \B((\bC^3\otimes \bC^2)\oplus (\bC^2\otimes \bC^3))
\qand
\T:=
\begin{pmatrix}
M_3\otimes I_2 & M_{3,2}\otimes M_{2,3}\\[2mm]
M_{2,3}\otimes M_{3,2} & M_2\otimes I_3
\end{pmatrix}.
\]
It can easily be verified that $(\R,\C)$, $(\S,\A)$ and $(\T,\B)$ are quantum graphs, and that $\S$ and $\T$ are also irreducibly acting as  $ \ca(\S)= M_7$ and $\ca(\T)=M_{12}$. Moreover, both $(\S,\A)$ and $(\T,\B)$ are full pullbacks of $(\R,\C)$ along the faithful unital $*$-homomorphisms
\[
\theta \colon \C\to \A,
\qquad
\theta(a\oplus b)=(I_2\otimes a)\oplus b,
\]
and
\[
\varrho \colon \C\to \B,
\qquad
\varrho(a\oplus b)=(I_3\otimes a)\oplus (I_2\otimes b).
\]
Hence
\[
\S=\R^{\leftarrow\theta}
\qand
\T=\R^{\leftarrow\varrho},
\]
and therefore Theorem~\ref{th:stronghom} yields that
$\S\sim_{\rm TRO}\T$.
\end{example}

As a consequence we obtain the characterisation of TRO-equivalence of graph operator systems in terms of classical pullback maps proved in \cite[Corollary 7.6]{EKT21}.

\begin{corollary}\cite[Corollary 7.6]{EKT21}\label{cor:classical}
Let $\G$ and $\H$ be graphs and let 
$\S_{\G} \subseteq M_n$ and 
$\S_{\H} \subseteq M_m$ 
be their graph operator systems.  
The following are equivalent:
\begin{enumerate}
\item $\S_{\G} \sim_{\rm TRO} \S_{\H}$.
\item The skeletons $\G^\circ$ and $\H^\circ$ are isomorphic. Equivalently, the graphs $\G$ and $\H$ are pullbacks of isomorphic graphs. 
\end{enumerate}
\end{corollary}

\begin{proof}
By Proposition \ref{P:Redgr} we have that $\G$ and $\H$ are pullbacks of isomorphic graphs if and only if their skeletons $\G^\circ$ and $\H^\circ$ are isomorphic. By Proposition \ref{P:irrgra} we get that $\S_\G$ and $\S_\H$ are irreducibly acting and hence we may invoke Theorem \ref{th:stronghom}. A use of Proposition \ref{prop:classical-weaver-pullback} then completes the proof.
\end{proof}

\subsection{TRO-equivalence of quantum graphs and their associated algebras}
In this subsection we investigate a simultaneous TRO-equivalence between the quantum graphs and their associated algebras. This provides a stronger notion of Morita equivalence.

\begin{theorem}\label{T:qgraphs_balanced}
Let $(\S,\A)$ and $(\T,\B)$ be quantum graphs acting on Hilbert spaces $H$ and $K$, respectively. The following are equivalent:
\begin{enumerate}
\item There is a non-degenerate TRO $\M\subseteq \B(H,K)$ which is also a $\B'$-$\A'$-bimodule implementing the TRO-equivalences
\[
\S \sim_{\rm TRO } \T \qand \A \sim_{\rm TRO } \B.
\]

\item There exists a faithful unital completely positive map $\vphi\colon\B\to \A$ such that
\[
\X_\vphi^*\T \X_\vphi\subseteq \S,\quad \X_\vphi \S \X_\vphi^*\subseteq \T 
\qand
\X_\vphi^*\B \X_\vphi\subseteq \A,\quad \X_\vphi \A \X_\vphi^*\subseteq \B.
\]
\end{enumerate}
\end{theorem}

\begin{proof}
\noindent
[(i) $\Rightarrow$ (ii)].
Assume that there exists a non-degenerate TRO $\M\subseteq \B(H,K)$ such that
\[
\M^* \T \M \subseteq \S,\quad \M \S \M^* \subseteq \T,
\qand
\M^* \B \M \subseteq \A,\quad \M \A \M^* \subseteq \B.
\]
Since $\B(H,K)$ is finite-dimensional, we may choose a basis
$\{b_i:i\in[r]\}$ for $\M$. By non-degeneracy of $\M$, we have that
$I_H\in [\M^*\M]$ and hence there exist operators $a_i\in \M$ for $i\in [r]$, such that
\[
I_H=\sum_{i=1}^r a_i^*b_i.
\]
Set
\[
a=\begin{bmatrix}a_1\\ \vdots\\ a_r\end{bmatrix}
\qand
b=\begin{bmatrix}b_1\\ \vdots\\ b_r\end{bmatrix}.
\]
Since $a^*b=I_H$ we have that $b\colon H\to K^{(r)}$ is injective, which implies that
\[
b^*b=\sum_{i=1}^r b_i^*b_i
\]
is injective on $H$ and hence also invertible. Since $b^*b\in [\M^*\M]$ and $[\M^*\M]$ is a unital
C*-algebra, we also have that $(b^*b)^{-1/2}\in [\M^*\M]$.
Set
\[
v_i=b_i(b^*b)^{-1/2}\in [\M\M^*\M]\subseteq \M \foral i\in [r].
\]
Then
\[
\sum_{i=1}^r v_i^*v_i
=(b^*b)^{-1/2}(\sum_{i=1}^r b_i^*b_i)(b^*b)^{-1/2}
=I_H.
\]
Thus $(v_i)_{i=1}^r$ is a family of Kraus operators defining a unital completely positive map
\[
\vphi\colon \B(K)\to \B(H)\text{ such that } \vphi(b)=\sum_{i=1}^r v_i^*bv_i.
\]
Since $v_i\in \M$, for every
$b\in \B$ and $i,j\in [r]$ we have
\[
v_i^* b\, v_j \in \M^*\B\M\subseteq \A,
\]
which yields that $\vphi(\B)\subseteq \A$.
Also, since $v_i\in \M$ and $\M$ is a $\B'$-$\A'$-bimodule we obtain that $\X_\vphi\subseteq \M$, and therefore we have
\[
\X_\vphi^*\T \X_\vphi\subseteq \M^*\T\M\subseteq \S
\qand
\X_\vphi \S \X_\vphi^*\subseteq \M\S\M^*\subseteq \T,
\]
and similarly
\[
\X_\vphi^*\B \X_\vphi\subseteq \M^*\B\M\subseteq \A
\qand
\X_\vphi \A \X_\vphi^*\subseteq \M\A\M^*\subseteq \B.
\]

Moreover, right multiplication by the invertible operator $(b^*b)^{-1/2}$ is a
vector space automorphism of $\B(H,K)$, and so
\[
\dim(\spn\{v_i:i\in [r]\})
=
\dim(\spn\{b_i(b^*b)^{-1/2}:i\in [r]\})
=
\dim(\spn\{b_i:i\in [r]\})
=r.
\]
Since $\spn\{v_i:i\in [r]\}\subseteq \M$ and $\M$ is $r$-dimensional it follows that
\[
\X_\vphi=\spn\{v_i:i\in [r]\}=\M.
\]
In particular, 
\[
[\X_\vphi H]=[\M H]=K,
\]
which by Lemma \ref{L:injectivity} is equivalent to $\vphi$ being faithful.

\medskip
\noindent
[(ii) $\Rightarrow$ (i)].
Suppose that
$\vphi\colon\B\to \A$ is a faithful unital completely positive map such that
\[
\X_\vphi^*\T\X_\vphi\subseteq \S,
\qquad
\X_\vphi \S \X_\vphi^*\subseteq \T 
\qand
\X_\vphi^*\B\X_\vphi\subseteq \A,
\qquad
\X_\vphi\A\X_\vphi^*\subseteq \B.
\]
By Lemma \ref{L:injectivity} we obtain that $\X_{\vphi}$ is non-degenerate. Hence, if we set $ \M = [\X_{\vphi} \ca(\X_{\vphi}^*\X_{\vphi})]$ we obtain a non-degenerate TRO that is moreover a $\B'$-$\A'$-bimodule, and implements the TRO-equivalences $\S\sim_{\rm TRO}\T$ and $\A\sim_{\rm TRO}\B$ by a use of Proposition \ref{prop:nondeg_TRO}.
\end{proof}

\begin{remark}
The two notions of Morita equivalence for quantum graphs that we consider in Theorem \ref{th:stronghom} and Theorem \ref{T:qgraphs_balanced} are in general distinct. For example, consider the quantum graphs $\S=M_n$ on $\A = M_n$ and $ \T= M_m$ on $\B=\bC I_m$. Then $ \S$ is clearly TRO-equivalent to $\T $ via $\M= M_{m,n}$; however $ \A$ and $\B$ are $\De$-equivalent but not TRO-equivalent. Indeed, if this was the case, we would have that $\bC I_n= \A' \cong \B'= M_m$ which is not possible for $ n,m>1$.    
\end{remark}

While the previous remark shows that in general the two notions of Morita equivalence are distinct, they coincide in the case of noncommutative graphs. In \cite[Proposition 7.2]{EKT21} it is shown that TRO-equivalence between two noncommutative graphs $\S \subseteq M_n$ and $\T \subseteq M_m$ is equivalent to the existence of a non-degenerate subspace $\X \subseteq M_{m,n} $  such that $\X^*\T \X \subseteq \S$ and $ \X \S \X^* \subseteq \T$. As shown in the next corollary the operator space $\X$ may be assumed to be merely non-degenerately acting.  In particular, TRO-equivalence implies the existence of a strong co-homomorphism between $\T$ and $\S$.

\begin{corollary} \label{T:ncgraphs}
Let $\S \subseteq M_n$ and $\T \subseteq M_m$ be noncommutative graphs.  
The following are equivalent:
\begin{enumerate}
\item $\S \sim_{\rm TRO} \T$;
\item There exists a non-degenerately acting subspace $ \X \subseteq M_{m,n}$ such that 
\[
\X^*\T \X \subseteq\S \qand \X\S\X^* \subseteq \T;
\]
\item There exists a (faithful) unital completely positive map $ \vphi \colon M_m \to M_n$ such that 
\[
\S = \T^{\leftarrow\vphi} \qand\T = \S^{\rightarrow\vphi}.
\]
\end{enumerate}
\end{corollary}

\begin{proof}
The equivalence of items (i) and (iii) follows from  Theorem \ref{T:qgraphs_balanced}  and  Theorem \ref{thm:ucppullback} while the equivalence of items (i) and (ii) follows from Proposition \ref{prop:nondeg_TRO}.
\end{proof}

We now see some examples in this case.

\begin{example}
Consider
\[
\S=\operatorname{span}\{I_2,E_{12},E_{21}\}\subseteq M_2
\qand
\A=M_2,
\]
and the faithful unital $*$-representation
\[
\theta \colon M_2\to M_4,\qquad \theta(x)=x\oplus x
= v_1^* x v_1 + v_2^* x v_2,
\]
where
\[
v_1=\begin{bmatrix} I_2 & 0 \end{bmatrix},
\qquad
v_2=\begin{bmatrix} 0 & I_2 \end{bmatrix}.
\]
Then
\[
\T:=\S^{\leftarrow\theta}
=
\operatorname{span}\{v_i^*sv_j:s\in\S, i,j=1,2\}
=
\S\otimes M_2.
\]
Moreover, $\T$ is a quantum graph on
\[
\theta(M_2)
=
\left\{
\begin{pmatrix}
x&0\\
0&x
\end{pmatrix}:x\in M_2
\right\}\subseteq M_4.
\]
Hence, by Corollary~\ref{T:ncgraphs} $\S\sim_{\rm TRO} \T$. In fact, amplifications of $ \S$ are the only quantum graphs on $M_m$ that are TRO-equivalent to $\S$ (see \cite[Example 3.11]{EKT21}).
\end{example}

\begin{example}
\rm Let $\B(H) = \B(K) = M_{n}$ for some $n \in \bN$. Then if $\Phi, \Psi: M_{n}\rightarrow M_{n}$ are $*$-automorphisms of $M_{n}$ (i.e., they are describing closed quantum dynamical systems on $M_{n}$), we have $\S_{\Phi} \sim_{\rm TRO} \S_{\Psi}$. Indeed, as every $*$-automorphism of $M_{n}$ is inner, the Kraus decomposition for both are given by unitary operators $u, \wt{u} \in M_{n}$ with $\Phi(a) = uau^{*}$ and $\Psi(a) = \wt{u}a\wt{u}^{*}$, so 
\[
\S_{\Phi} = \spn\{u^{*}u\} = \mathbb{C}I_{n} = \spn\{\wt{u}^{*}\wt{u}\} = \S_{\Psi},
\]
with TRO-equivalence following trivially. 
\end{example}

\begin{example}
\rm A \textit{depolarizing channel} $\Phi_{\lambda}: M_{n}\rightarrow M_{n}$ is a quantum channel of the form 
\[
\Phi_{\lambda}(\rho) = (1-\lambda)\rho+\frac{\lambda}{n}I_{n} \text{ for some } 0 \leq \lambda \leq 1+\frac{1}{n^{2}-1},
\]
and it is \textit{completely depolarizing} if $\lambda = 1$. Let $\Phi\colon M_{n}\rightarrow M_{n}$ and $\Psi\colon M_{m}\rightarrow M_{m}$ be completely depolarizing channels on their respective systems.
Then $\S_{\Phi}\sim_{\rm TRO} \S_{\Psi}$.

Indeed, as any two states are confusable the confusability graphs for $\Phi$ and $\Psi$ are easily seen to be $\S_{\Phi} = M_{n}$ and $\S_{\Psi} = M_{m}$ and hence $\M = M_{m, n}$ is a non-degenerate TRO implementing TRO-equivalence between $\S_{\Phi}$ and $\S_{\Psi}$. 

In fact, if $\N: [m]\rightarrow [k]$ is any classical channel for which $\N(j|i) \neq 0$ for all $i \in [m]$ and $j \in [k]$, then $\S_{\Phi} \sim_{\rm TRO}\S_{\N}$ for any completely depolarizing channel $\Phi$. This follows from the previous argument, in conjunction with the fact that the Kraus operators for $\N$ (considered as a quantum channel) are $\sqrt{\N(j|i)}\epsilon_{ji}$, with all non-zero for every choice of $(i, j) \in [m]\times [k]$. 
\end{example}

Finally, we end this section by showing that this stronger notion of Morita equivalence, in the sense of Theorem \ref{T:qgraphs_balanced}, collapses into isomorphism in the case of graph operator systems.

\begin{remark}\label{R:graphstrtro}
Let $\G$ and $\H$ be graphs on $n$ and $m$ vertices, respectively, and suppose that $\S_\G\sim_{\rm TRO}\S_\H$ and $D_n\sim_{\rm TRO}D_m$ via a non-degenerate TRO $\M$ which is a $D_m$-$D_n$-bimodule. In this case, we have
\[
[\M^* \M]=[\M^* D_m \M]=D_n \qand [\M\M^*]=[\M D_n \M^*]=D_m,
\]
and by \cite[Theorem 3.2]{Ele12} we obtain a $*$-isomorphism $\pi \colon D_n \to D_m$ such that
\[
\M=\{x\in M_{m,n}: xa=\pi(a)x \foral a\in D_n\}.
\]
In particular, $n=m$ and $\pi$ is a conjugation by a permutation matrix $u\in M_n$. It then follows that $\M=D_n u=uD_n$ and hence
\[
\S_\H=[\M^* \S_\G \M]= [u^* D_n \S_\G D_n u]= u^* \S_\G u.
\]
we conclude that $u^*\S_G u=\S_\H$ and \cite[Proposition 3.1]{OP14} yields that $\G\cong \H$.
\end{remark}


\section{Noncommutative graph parameters} \label{sec:nonc}
In this section, we specialize to the case of noncommutative graphs, i.e., quantum graphs on $M_n$ for $ n \in \bN$, and investigate how TRO-equivalence of quantum graphs interacts with various noncommutative graph parameters. The following definitions are taken from \cite{BTW21,DSW13,Pau}; we indicate the source when appropriate.

Let $H$ be a finite-dimensional Hilbert space with $\dim H=k$ and let $\S\subseteq\B(H)$ be an operator system. Following \cite{BTW21,DSW13}, a set of mutually orthogonal unit vectors
\[
\{\xi_{\ell}\}_{\ell=1}^{n} \subseteq H
\]
is called $\S$-\emph{independent} if
\[
\xi_i\xi_j^* \in \S^\perp, 
\qquad i\neq j,
\]
and  $\S$-\emph{clique} if 
\[
\xi_i\xi_j^* \in \S, 
\qquad i\neq j.
\]

We recall several numerical parameters associated with $\S$, which are noncommutative analogues of classical graph parameters.
The \emph{independence number} of $\S$ is
\[
\alpha(\S)
:=
\max\Big\{
n\in\mathbb{N} :
\exists\, \{\xi_i\}_{i=1}^{n} \subseteq H
\text{ forming an $\S$-independent set}
\Big\}.
\]
The \emph{clique number} of $\S$ is
\[
\omega(\S)
:=
\max\Big\{
n\in\mathbb{N} :
\exists\, \{\xi_i\}_{i=1}^{n} \subseteq H
\text{ forming an $\S$-clique}
\Big\}.
\]
The fact that $\al$ is a quantization of the classical independence number was shown in \cite{DSW13}, and that $\omega$ is a quantisation of the classical clique number was shown \cite[Corollary 3.9]{BTW21}.
The \emph{minimal chromatic number} of $\S$, denoted $\chi_{0}(\S)$, is the least $n\in\mathbb{N}$ such that there exist
a basis $\{\xi_{i}\}_{i=1}^{k}$ of $H$ and a partition
\[
\{1,\dots,k\}=S_{1}\cup\cdots\cup S_{n},
\]
for which
\[
\xi_i\xi_j^* \in \S^\perp,
\text{ when }
i\neq j,\; i,j\in S_{\ell},\; \ell=1,\dots,n .
\]
In \cite[Theorem 14]{KM19} (see also \cite{Pau}) it is shown that this is quantization of the classical chromatic number.

\begin{remark}
Note that the definition of the minimal chromatic number $\chi_0$ is a slightly modified version of the chromatic number as defined in  \cite{KM19} and \cite{Pau}. In the original definition, the existence of an orthonormal basis is required, whereas here we require only a basis for the space. 
\end{remark}

For $n \in \bN$, we use 
\[
\S^{\otimes n} := \underbrace{\S\otimes \cdots \otimes \S}_{n \; {\rm times}} \subseteq \B(H^{n}),
\]
to denote the $n$-fold tensor product operator system of $\S$ with itself (where we use the minimal operator system tensor product). 

Let $\Phi\colon \B(H)\rightarrow \B(K)$ be a quantum channel between finite-dimensional Hilbert spaces, and $\S_{\Phi}$ be the noncommutative graph of $\Phi$. Following \cite{DSW13} the \textit{Shannon capacity of $\Phi$} is defined as
\[
c_{0}(\Phi) = \lim\limits_{n\rightarrow \infty}\sqrt[n]{\alpha(\S_{\Phi}^{\otimes n})}.
\]
By virtue of \cite[Lemma 2]{DUAN09}, every noncommutative graph $\S\subseteq \B(H)$ with $H$ finite-dimensional, is the noncommutative graph for a quantum channel $\Phi$. Thus, we define the Shannon capacity $c_{0}(\S):=c_{0}(\Phi)$ where $\S=\S_\Phi$.

In the discussion following \cite[Lemma 4.1]{BTW21} and in the discussion preceding \cite[Theorem 5.1]{BTW21} the authors define two quantizations of the classical Lov\'asz number of a graph as defined in \cite{Lov79}. See \cite[Corollary 4.6 and Proposition 5.3]{BTW21}.
We will only use the equivalent formulations of $\vartheta$ and $\wh{\vartheta}$ obtained in \cite[Theorem 5.1]{BTW21}.
Namely, for $H$ being a finite-dimensional Hilbert space, and $\S \subseteq \B(H)$ being an operator system we have
\begin{align*}
&\wh{\vartheta}(\S)^{-1} = \sup\{\inf\{\|\Phi(\rho)\|:\rho  \text { a state of } \B(H) \}:  \Phi \in \fC(\S)\},\\
&\vartheta(\S)^{-1} = \inf\{\sup\{\|\Phi(\rho)\|: \Phi \in \fC(\S)\}: \rho \text { a state of } \B(H) \}.
\end{align*}

Following \cite{LPT18}, for an arbitrary noncommutative graph $\S \subseteq \B(H)$, the \textit{quantum complexity} and \textit{quantum subcomplexity} are defined as
\begin{align*}
&\ga(\S) := \min\{k \in \bN: \text{there exists } \Phi\colon \B(H)\rightarrow M_{k} \text{ with } \S_{\Phi} = \S\},\\
&\beta(\S) := \min\{k \in \bN: \text{there exists } \Phi \colon \B(H)\rightarrow M_{k} \text{ with } \Phi \in \fC(\S)\}. 
\end{align*}
See \cite[Theorem IV.10]{LPT18}.

Finally, a noncommutative graph analogue of the \textit{Haemers bound} was introduced in \cite{GL20}, see \cite[Proposition 5]{GL20}. For a noncommutative graph $\S \subseteq \B(H)$, the Haemers bound $\mathfrak{H}(\S)$ is defined as
\[
\fH(\S) := \min\{k \in \mathbb{N}: \text{ there exists } \Phi\colon \B(H)\rightarrow M_{k} \text{ trace-preserving with } \wt{\S}_{\Phi} \subseteq \S\}. 
\]

\begin{remark}\label{R:chrqli}
It is not true that $\chi_{0}$ and $\omega$ are invariant under TRO-equivalence. This can be shown in the setting of graph operator systems. Indeed, consider two complete graphs $\K_n$ and $\K_m$ with $n\neq m$. Then $\S_{\K_n}=M_n$ and $\S_{\K_m}=M_m$ are TRO-equivalent via $\M=M_{n,m}$ and
\[
\omega(\S_{\K_n})=\chi_{0}(\S_{\K_n}) = n\neq m = \chi_{0}(\S_{\K_m})=\omega(\S_{\K_m}).
\] 
\end{remark}

We now prove the main result of this section. 

\begin{theorem}\label{th:nc_graph_parameter_invariance}
The parameters $\alpha, c_{0}, \vartheta, \wh{\vartheta}, \fH, \beta$ and $\ga$ are invariant under TRO-equivalence of noncommutative graphs.
\end{theorem}

\begin{proof}
Let $\S \subseteq \B(H)$ and $\T \subseteq \B(K)$  be irreducibly acting noncommutative graphs on finite-dimensional Hilbert spaces $H, K$  and suppose that $\S\sim_{\rm TRO} \T$.
By Corollary~\ref{T:ncgraphs} we may pick a  unital completely positive map $\vphi\colon \B(K)\rightarrow \B(H)$ with Kraus operators $(v_i)_{i=1}^r\subseteq \B(H,K)$ that satisfy
\[
v_{i}\S v_{j}^* \subseteq \T \qand v_{i}^*\T v_{j} \subseteq \S \foral i, j \in [r].
\]
Consider the column isometry $V:= (v_{1}, \dots, v_{r})$.

First, we show that $\alpha(\S) = \alpha(\T)$. Let $\{\xi_{\ell}\}_{\ell=1}^{m}$ be an $\S$-independent set. As $V$ is a column isometry, we claim that $\{V\xi_{\ell}\}_{\ell=1}^{m}$ is an $M_{r}(\T)$ independent set. Indeed, it follows that $\{V\xi_{\ell}\}_{\ell=1}^{m}$ is an orthonormal set, and by taking orthogonal complements 
\begin{gather}\label{eqn_hs_orth_inclusions}
v_{i}\S^{\perp}v_{j}^* \subseteq \T^{\perp} \qand v_{i}^*\T^{\perp}v_{j}  \subseteq \S^{\perp},
\end{gather}
for all $i, j \in [r]$. As $\{\xi_{\ell}\}_{\ell=1}^{m}$ is an $\S$-independent set, if $\ell \neq p$, then $\xi_{\ell}\xi_{p}^* \in \S^{\perp}$. By (\ref{eqn_hs_orth_inclusions}), we then have that $v_{i}\xi_{\ell}\xi_{p}^*v_{j}^* \in \T^{\perp}$ for $\ell \neq p$. Thus, $\{V\xi_{\ell}\}_{\ell=1}^{m}$ is an $M_{r}(\T)$-independent set. This implies $\alpha(\S) \leq \alpha(M_{r}(\T))$. Finally, using \cite[Remark IV.6]{LPT18}, we may argue $\alpha(M_{r}(\T)) = \alpha(\T)$, showing $\alpha(\S) \leq \alpha(\T)$. A symmetric argument shows that $\alpha(\T) \leq \alpha(\S)$, and thus $\alpha(\S) = \alpha(\T)$ as claimed.

We now prove that $c_{0}(\S) = c_{0}(\T)$. We rely on the (readily verified) fact that if $\S \sim_{\rm TRO} \T$, then $\S^{\otimes n} \sim_{\rm TRO} \T^{\otimes n}$  for each $n \in \bN$. By the previous result, we have that $\alpha(\S^{\otimes n}) = \alpha(\T^{\otimes n})$ for every $n \in \mathbb{N}$. Thus,
\[
c_{0}(\S) = \lim\limits_{n\rightarrow \infty}\sqrt[n]{\alpha(\S^{\otimes n})} = \lim\limits_{n\rightarrow \infty}\sqrt[n]{\alpha(\T^{\otimes n})} = c_{0}(\T).
\]

Next we prove that $\beta(\S) = \beta(\T)$ and $\ga(\S) = \ga(\T)$.
Assume that $\beta(\T) = n$ and let $\Phi\colon \B(K)\rightarrow M_{n}$ be a quantum channel in $\fC(\T)$ with Kraus decomposition 
\[
\Phi(a) = \sum_{i=1}^{\ell}w_{i}aw_{i}^* \foral a \in \B(K).
\]
The composition $\Phi\circ\vphi^{*} : \B(H)\rightarrow M_{n}$ with Kraus decomposition
\[
(\Phi\circ \vphi^{*})(a) = \sum\limits_{i=1}^{\ell}\sum\limits_{j=1}^{r}w_{i} v_{j} a v_{j}^* w_{i}^*, \foral a  \in \B(H),
\]
is another quantum channel. Furthermore, we claim $\Phi\circ \vphi^{*} \in \fC(\S)$. Indeed, as $\Phi \in \mathfrak{C}(\T)$, we have that $w_{j}^*w_{i} \in \T$ for each $i, j \in [\ell]$. Therefore, 
\[
\S_{\Phi\circ \vphi^{*}} = \spn\{v_{p}^*w_{j}^*w_{i}v_{q}: \; i, j \in [\ell], p, q \in [r]\} \subseteq \spn\{v_{p}^*\T v_{q}: \; p, q \in [r]\} \subseteq \S,
\]
by the inclusion relations on $\vphi$. Therefore, $\Phi\circ \vphi^{*} \in \fC(\S)$ as claimed, and so $\beta(\S) \leq \beta(\T)$. Under TRO-equivalence, we may use a symmetric argument to conclude $\beta(\S) = \beta(\T)$. The argument for $\ga$ is analogous.

Finally, we show that $\vartheta(\S)=\vartheta(\T)$,
$\wh{\vartheta}(\S)=\wh{\vartheta}(\T)$ and $\fH(\S) = \fH(\T)$.
The dual map 
\[
\vphi^*\colon \B(H)\to \B(K); \qquad \vphi^*(a)=\sum_{i=1}^{r}v_{i}av_{i}^* \foral a \in \B(K),
\]
is completely positive and trace-preserving and its Kraus operators satisfy $v_{i}^*\T v_{j} \subseteq \S$ for all $i, j \in [r]$. In \cite[Proposition 6.1]{BTW21} it is shown that the quantities $ \vartheta$ and $ \wh \vartheta$ are monotone under such maps. Hence, $\vartheta(\T)\le \vartheta(\S)
$ and $\wh{\vartheta}(\T)\le \wh{\vartheta}(\S)$. 
Since TRO-equivalence is symmetric, the same argument in the opposite direction yields $\vartheta(\S)=\vartheta(\T)$ and
$\wh{\vartheta}(\S)=\wh{\vartheta}(\T)$. An analogous argument and \cite[Proposition 12]{GL20} yield $\fH(\S) = \fH(\T)$, and the proof is complete.
\end{proof}

In the case of graph operator systems we obtain the following.

\begin{corollary}\label{cor:graph_parameter_invariance}
The graph parameters $\alpha, \vartheta, \fH, \beta$ and $\ga$ are invariant under TRO-equivalence of the corresponding graph operator systems. 
\end{corollary}

\begin{remark}
We note that it can be the case that two graphs have equal values for several standard graph parameters, yet are not TRO-equivalent. In what follows, we rely on Corollary~\ref{cor:graph_parameter_invariance}. 
Consider the graphs $\G$ and $\H$ on four vertices shown in Figure~\ref{fig:g1}. 
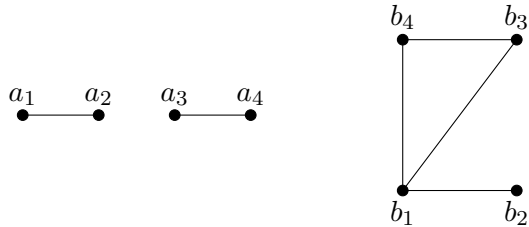
\begin{figure}[!ht]
\centering
\begin{tikzpicture}
\filldraw[black] (-2, 1) circle (2pt) node[anchor=south] {$a_{1}$};
\filldraw[black] (-1, 1) circle (2pt) node[anchor=south] {$a_{2}$};
\filldraw[black] (0, 1) circle (2pt) node[anchor=south] {$a_{3}$};
\filldraw[black] (1, 1) circle (2pt) node[anchor=south] {$a_{4}$};
\filldraw[black] (3, 0) circle (2pt) node[anchor=north] {$b_{1}$};
\filldraw[black] (4.5, 0) circle (2pt) node[anchor=north] {$b_{2}$};
\filldraw[black] (4.5, 2) circle (2pt) node[anchor=south] {$b_{3}$};
\filldraw[black] (3, 2) circle (2pt) node[anchor=south] {$b_{4}$};
\draw (-2, 1)--(-1, 1);
\draw (0, 1)--(1, 1);

\draw (3, 0)--(4.5, 0);
\draw (3, 0)--(3, 2);
\draw (3, 0)--(4.5, 2);
\draw (3, 2)--(4.5, 2);
\end{tikzpicture} 
\caption{The graphs $\G$ and $\H$, respectively}
\label{fig:g1}
\end{figure}

It is readily verified that $\alpha(\G)=\alpha(\H)=2$.
Moreover, we claim that
\[
\vartheta(\G)=\vartheta(\H)=2, \quad \beta(\mathcal{G}) = \beta(\mathcal{H}) = 2, \quad \text{and} \quad \fH(\G)=\fH(\H)=2.
\]
Indeed, for both graphs the complements $\ol{\G}$ and $\ol{\H}$ are bipartite, and hence have chromatic number $2$. Using the inequalities
\[
\alpha(G)\le \vartheta(G)\le \chi_0(\ol{G}),
\]
we obtain
\[
2=\alpha(\G)\le \vartheta(\G)\le 2,
\quad 
2=\alpha(\H)\le \vartheta(\H)\le 2,
\]
which implies $\vartheta(\G)=\vartheta(\H)=2$. The same conclusion holds for the Haemers bound, since
\[
\alpha(G)\le \fH(G)\le \vartheta(G),
\]
and both bounds coincide. To see the claim on quantum sub-complexity, we rely on the equivalence $\beta(\mathcal{G}) = \xi(\ol{\mathcal{G}})$ for any finite simple graph, where $\xi$ denotes the \textit{orthogonal rank} (see \cite{SS12}). An explicit orthogonal representation for $\ol{\mathcal{G}}$ is given by assigning $e_{1}$ to $a_{1}, a_{2}$ and $e_{2}$ to $a_{3}, a_{4}$ (where $e_{1}, e_{2}$ are the canonical orthonormal basis elements for $\mathbb{C}^{2}$). Furthermore, it is easily seen that $\xi(\ol{\mathcal{G}}) > 1$, and thus $\xi(\ol{\mathcal{G}}) = 2$. Similarly, assigning $e_{1}$ to $b_{1}, b_{2}$ and $e_{2}$ to $b_{3}, b_{4}$ gives an orthogonal representation for $\ol{\mathcal{H}}$ which shows $\xi(\ol{\mathcal{H}}) = 2$. Thus, $\beta(\mathcal{G}) = \beta(\mathcal{H}) = 2$. 

We also claim that
\[
\ga(\ol{\G})=\ga(\ol{\H})=3.
\]
To see that $\ga(\ol{\H})=3$, note that $\ol{\H}$ contains an isolated vertex. Thus, in any feasible tuple $x=(x_1,x_2,x_3,x_4)$ with non-zero $x_i\in \mathbb{C}^k$, the vector corresponding to this vertex must be orthogonal to all others. The remaining three vertices form a copy of $P_3$, which imposes additional orthogonality constraints, forcing at least three mutually orthogonal non-zero vectors. Hence $k\ge 3$. It is easily verified that this bound is attained, and so $\ga(\ol{\H})=3$.
For $\ol{\G}$, note that it contains $P_4$ as a subgraph, and $\ga(P_4)=3$. Thus $\ga(\ol{\G})\ge 3$. A modification of the tuple used for $P_4$ yields a feasible tuple in $\mathbb{C}^3$ for $\ol{\G}$, and hence $\ga(\ol{\G})=3$. 

We claim that $\S_{\G} \not\sim_{\rm TRO} \S_{\H}$. Indeed, the graph $\G$ consists of two disjoint edges, so its vertices split into two equivalence classes under the true-twin equivalence relation $\approx$, each of size $2$. Hence the skeleton $\G^\circ$ consists of two isolated vertices. 
On the other hand, in $\H$ the vertices $\{b_2\}$, $\{b_3,b_4\}$ and $\{b_1\}$ form three distinct equivalence classes under $\approx$, and thus $\H^\circ$ is isomorphic to the path $P_3$. Since $\G^\circ \not\cong \H^\circ$, Corollary~\ref{cor:classical} implies that $\S_{\G} \not\sim_{\rm TRO} \S_{\H}$.
Similarly, for the complements $\ol{\G}$ and $\ol{\H}$, one checks that no two distinct vertices have identical closed neighbourhoods, and hence both graphs have trivial twin classes. Therefore, $\ol{\G}^\circ=\ol{\G} $
and $\ol{\H}^\circ=\ol{\H}$.
Since these graphs are not isomorphic, Corollary~\ref{cor:classical} yields
\( 
\S_{\ol{\G}} \not\sim_{\rm TRO} \S_{\ol{\H}}.
\)
Thus, $\G$ and $\H$ (as well as their complements) may share the same values of $\alpha$, $\vartheta$, $\fH$, $\beta$ and $\ga$, yet are not TRO-equivalent.
\end{remark}

The following result is folklore. We obtain it as a corollary of Theorem \ref{th:nc_graph_parameter_invariance}.

\begin{corollary}\label{C:fullmattro}
Let $H, K$ be finite-dimensional Hilbert spaces and let $\S \subseteq \B(H)$ be a noncommutative graph. Then $\S \sim_{\rm TRO} \B(K)$ if and only if $\S = \B(H)$. 
\end{corollary}

\begin{proof}
If $\S = \B(H)$, then $\M = \B(H, K)$ is a non-degenerate TRO that implements TRO-equivalence between $\B(H)$ and $\B(K)$. Assume now that $\S \sim_{\rm TRO} \B(K)$. For any quantum channel $\Phi\colon \B(K)\rightarrow M_{d}$ we have $\S_{\Phi} \subseteq \B(K)$, and thus $\Phi \in \mathfrak{C}(\B(K))$. Using \cite[Theorem 5.1]{BTW21} we get $\wh{\vartheta}(\B(K)) = 1$ and hence Theorem \ref{th:nc_graph_parameter_invariance} yields that $\wh{\vartheta}(\S) = 1$. Finally, by \cite[Proposition 6.7]{BTW21} we have $\S = \B(H)$, as required.
\end{proof}

We finish the section with an application to connectedness in the quantum graph setting. Connectedness for quantum graphs was introduced in \cite{CDS21} in the noncommutative graph case and in \cite{Mat24} in the quantum adjacency setting.  The  definition that we use comes from the work of  \cite{CGW25} (see Theorem 3.4 and  Lemma 3.5 therein)  in which the different models were unified. 

Let $(\S,\A)$ be a quantum graph acting on a Hilbert space $H$. We say that the quantum graph
$(\S,\A)$ is \emph{connected} if
\[
\ca(\S)= \B(H).
\]

\begin{proposition}\label{P:connec}
Let $(\S,\A)$ and $ (\T,\B)$ be two irreducibly acting quantum graphs. Assume that
\[
\S \sim_{\Delta} \T.
\]
Then $(\S,\A)$ is connected if and only if $(\T,\B)$ is connected.
\end{proposition}

\begin{proof}
By \cite[Proposition 7.5]{EKT21},  we have $ \ca(\S) \sim_{\rm TRO} \ca(\T)$. Hence, by Corollary \ref{C:fullmattro} we obtain that $\ca(\S)= \B(H)$ if and only if $\ca(\T)= \B(K)$.
\end{proof}

In particular, for graph operator systems we have the following.

\begin{corollary}
Let $\G$ and $\H$ be graphs and suppose that $\S_{\G}\sim_{\Delta}\S_{\H}$. Then $\G$ is connected if and only if $\H$ is connected.
\end{corollary}

\begin{proof}
It follows from \cite[Corollary 3.4]{CDS21} that $\G$ is connected if and only if $\S_{\G}$ is connected.
The result now follows from Proposition \ref{P:connec}.
\end{proof}


\subsection*{Acknowledgments} The authors would like to thank John Byrne for helpful comments, in particular for references on true-twin equivalence classes of a graph. The authors are grateful to Michalis Anoussis, George Eleftherakis, Mahya Ghandehari, Evgenios Kakariadis, Aristides Katavolos and Ivan G. Todorov for their helpful discussions and suggestions.

The first, third and last named authors acknowledge that this research work was supported within the framework of the National Recovery and Resilience Plan Greece 2.0, funded by the European Union - NextGenerationEU (Implementation Body: HFRI. Project name: Noncommutative Analysis: Operator Systems and Nonlocality. HFRI Project Number: 015825).
The second named author acknowledges support from the U.S. Department of Defense, Basic Research Office, under Vannevar Bush Faculty Fellowship grant N00014-21-1-2946, managed by the Office of Naval Research. 
The last named author acknowledges that this research work was supported by the Hellenic Foundation for Research and Innovation (HFRI) under the 5th Call for HFRI PhD Fellowships (Fellowship Number: 19145).
This material is based upon work supported by the Swedish Research Council under grant no. 2021-06594 while the first, second and last named authors were in residence at Institut Mittag-Leffler in Djursholm, Sweden during the Spring semester 2026.
The authors would like to thank the institute for their hospitality.

\subsection*{Open acess statement}
For the purpose of open access, the authors have applied a Creative
Commons Attribution (CC BY) license to any Author Accepted Manuscript (AAM) version
arising.


\end{document}